\documentclass[a4paper, 10pt, final]{article}

\usepackage{eucal}	
\usepackage{amsmath}
\usepackage{amsfonts}
\usepackage[utf8]{inputenc}
\usepackage{amsthm}
\usepackage{amssymb}
\usepackage{graphicx}
\usepackage{graphics}
\usepackage{bm}
\usepackage[perpage]{footmisc}

\usepackage{dsfont}
\usepackage{color}
\usepackage{enumerate}
\usepackage[font=footnotesize]{caption}
\usepackage[all]{xy}
\usepackage{xypic}
\allowdisplaybreaks
\usepackage{enumerate}		

\usepackage{mathtools}		
\usepackage{mathrsfs}		
\usepackage{eucal}			
\usepackage{braket}	
\usepackage[colorlinks=true, linkcolor=black, citecolor=black,
urlcolor=blue]{hyperref}		
\mathtoolsset{showonlyrefs}
\usepackage{mparhack}		
\usepackage{relsize}			
\usepackage{a4wide}		
\usepackage{booktabs}		
\usepackage{multirow}		
\usepackage{caption}		
\usepackage{rotating}
\usepackage{varioref}		
\usepackage{footmisc}

\usepackage{fixltx2e}		

\newcommand{\trasp}[1]{{#1}^\mathsf{T}}			
\newcommand{\traspinv}[1]{{#1}^\mathsf{-T}}	

\numberwithin{equation}{section}

\theoremstyle{plain}
\newtheorem*{thm*}{Theorem}
\newtheorem{thm}{Theorem}[section]

\newtheorem{cor}[thm]{Corollary}
\newtheorem{lem}[thm]{Lemma}
\newtheorem{prop}[thm]{Proposition}
\theoremstyle{definition}

\newtheorem{dfn}[thm]{Definition}
\theoremstyle{remark}
\newtheorem{rem}[thm]{Remark}
					
\renewcommand{\=}{\coloneqq}			

\newcommand{\N}{\mathds{N}}

\newcommand{\X}{\mathbb{X}}

\newcommand{\R}{\mathds{R}}

\newcommand{\eps}{\varepsilon}
\renewcommand{\S}{\mathbb{S}}

\newcommand{\OO}{\mathrm{O}}
\newcommand{\C}{\mathds{C}}

\newcommand{\SO}{\mathrm{SO}}

\renewcommand{\O}{\mathrm{O}}

\newcommand{\Mat}{\mathrm{Mat}}

\newcommand{\Id}{I}

\newcommand{\im}{\mathrm{im\,}}
\newcommand{\mscal}[2]{\langle#1,#2\rangle_{M}}
\newcommand{\mnorm}[1]{|#1|_M}

\DeclareMathOperator{\diag}{diag}

\newcommand{\B}{\mathcal B}

\newcommand{\nullity}[1]{n^0(#1)}		
\newcommand{\iMor}[1]{n^{-}\left(#1\right)}
\newcommand{\extiMor}[1]{\mathring{n}^{-}(#1)}
\newcommand{\extcoiMor}[1]{\mathring{n}^{+}(#1)}
 \newcommand{\coiMor}[1]{n^{+}(#1)}
 \newcommand{\Euc}[1]{\mathrm{Euc}(#1)}

 \renewcommand{\B}{\mathscr{B}}		
\renewcommand{\H}{\mathcal{H}}	

\newcommand{\Bsa}{\B^{\textup{sa}}}

\renewcommand{\=}{\coloneqq}			
\DeclareMathOperator{\spfl}{sf}			
\DeclareMathOperator{\sgn}{\mathrm{sgn}}		

\usepackage{bm}
\newcommand{\simbolovettore}[1]{{\boldsymbol{#1}}}

\newcommand{\ve}{\simbolovettore{e}}

\newcommand{\vi}{\simbolovettore{i}}

\newcommand{\vj}{\simbolovettore{j}}
\newcommand{\vk}{\simbolovettore{k}}

\newcommand{\zero}{\boldsymbol{0}}

\newcommand{\rvline}{\hspace*{-\arraycolsep}\vline\hspace*{-\arraycolsep}}

\title{Spectral stability, spectral flow and circular relative equilibria for the  Newtonian $n$-body problem}
\author{Luca Asselle, Alessandro Portaluri, Li Wu}
\date{\today}

\date{\today}
\begin{document}
 \maketitle

 \begin{abstract}
For the Newtonian (gravitational) $n$-body problem in the Euclidean $d$-dimensional space, $d\ge 2$, the simplest possible periodic solutions are provided by circular relative equilibria, (RE) for short, namely solutions in which each body rigidly rotates about the center of mass and the configuration of the whole system is  constant in time and central (or, more generally, balanced) configuration. 
For $d\le  3$, the only possible (RE) are planar, but in dimension four it is possible to get truly four dimensional (RE). 

A classical problem in celestial mechanics aims at relating the (in-)stability properties of a (RE) to the index properties of the central (or, more generally, balanced) configuration generating it. In this paper, we provide sufficient 
conditions that imply the spectral instability of planar and non-planar (RE) in $\R^4$ generated by a central configuration, 
thus answering some of the questions raised in \cite[Page 63]{Moe14}. As a corollary, we retrieve 
a classical result of Hu and Sun \cite{HS09} on the linear instability of planar (RE) whose generating central configuration is non-degenerate and has odd Morse index, and fix a gap in the 
statement of  \cite[Theorem 1]{BJP14} about the spectral instability of planar (RE) whose (possibly degenerate) generating central configuration has odd Morse index. 

The key ingredients are a new formula of independent interest that allows to compute the spectral flow of a path of symmetric matrices having degenerate starting point, and 
a symplectic decomposition of the phase space of the linearized Hamiltonian system along a given (RE) which is inspired by Meyer and Schmidt's planar decomposition \cite{MS05} and which allows us 
to rule out the uninteresting part of the dynamics corresponding to the translational and (partially) to the rotational symmetry of the problem. 

\vspace{2mm}

\end{abstract}

\tableofcontents

%

\section{Introduction and description of the problem}

The Newtonian $n$-body problem concerns the study of the dynamics of $n$ point particles with positions $q_i \in \R^d$, $d\ge 2$, and  masses $m_i>0$, moving under the influence of their mutual gravitational attraction. Newton's law of motion for the gravitational $n$-body problem reads
\begin{equation}\label{eq:Newton-law}
m_i\ddot q_i = \sum_{j \neq i} \dfrac{m_im_j(q_i-q_j)}{|q_i-q_j|^3}
\end{equation}
where $|q_i-q_j|$ is the Euclidean distance between the $i$-th and $j$-th particle (the gravitational constant being normalized to $G=1$). Let $q= (q_1, \ldots, q_n) \in \R^{nd}$ be the {\sc configuration vector\/} and let 
\begin{equation}\label{eq:potential}
U(q)=\sum_{i <j}\dfrac{m_im_j}{|q_i-q_j|}
\end{equation}
be the {\sc Newtonian potential.\/} One readily sees that the RHS of~\eqref{eq:Newton-law}, that is the total force acting on the $i$-th particle, can be written as $F_i=\nabla_i U$, 
where $\nabla_i$ denotes the gradient with respect to the $d$ components of the $i$-th particle. Then Equation~\eqref{eq:Newton-law} can be written equivalently as
\begin{equation}\label{eq:Newton-cpt}
	M \ddot q= \nabla U(q) ,
\end{equation}
where $\nabla $ is the $dn$-dimensional gradient and $M:=\diag(m_1 \Id_d, \ldots, m_n \Id_d)$ is the {\sc mass matrix\/}. Equation~\eqref{eq:Newton-cpt} defines a system of second-order ordinary differential equations on $\mathbb X:= \R^{nd} \setminus\Delta$, where
$$\Delta:=\Big \{q \in \R^{nd}\ \Big |\ U(q)=+\infty\Big \}= \Big \{q\in \R^{nd}\ \Big |\ q_i=q_j, \text{for some}\ i\neq j\Big \}$$
is the {\sc collision set\/}. Also, Equation~\eqref{eq:Newton-cpt} is invariant under translations and rotations. Symmetries under translations yield via Noether's theorem to the conservation of the {\sc total momentum\/}. This implies that the center of mass has an inertial motion, and for this reason one can always assume that the center of mass be fixed at the origin. Invariance under rotations yields instead to the conservation 
of the {\sc angular momentum} (more about this can be found in the appendix). 

Among all configurations of the system, a special role is played by the so-called {\sc central configurations\/}, that is configuration vectors $q\in\mathbb X$ satysfying
\begin{equation}\label{eq:cc}
\nabla U(q)+  \frac{U(q)}{\langle Mq,q\rangle} Mq=0.
\end{equation}
In other words, central configurations are special arrangements of the particles in which the acceleration of each mass points towards the center of mass and is proportional to the distance to the center of mass. As it is nowadays well-known, central configurations play a key role in the understanding of the dynamics of the $n$-body problem for many reasons: they generate explicit periodic solutions of~\eqref{eq:Newton-cpt}, they govern the behavior of colliding solutions near collisions, they mark changes in the topology of the integral manifolds, etc.
 
The first solutions of~\eqref{eq:cc} were discovered by Euler, who determined all collinear solutions for the $3$-body problem, and Lagrange, who showed that up to symmetries the only non-collinear 
central configuration for the $3$-body problem is given by the equilateral triangle. For larger values of $n$ it is impossible to solve~\eqref{eq:cc} explicitly. 
Nevertheless, one can still say many things about central configurations by exploiting their variational characterization as critical points of the restriction $\widehat U$ of the Newtonian potential $U$ to the inertia ellipsoid 
\[
\S= \Set{q \in \R^{nd}| \sum_{i=1}^n m_i q_i=0,\  \langle Mq,q \rangle =1}.
\]
Indeed, since~\eqref{eq:cc} is invariant under scalings, it is not restrictive to assume that $\langle Mq,q\rangle =1$, that is that the central configuration be {\sc normalized} (hereafter, whenever talking about central configurations we implicitly assume that the condition $\langle Mq,q\rangle =1$ hold).

Since $\widehat U$ is $\SO(d)$-invariant (actually, $\O(d)$-invariant), central configurations are never isolated as critical points of $\widehat U$ but rather come in $\SO(d)$-families. In particular, they are 
always degenerate as critical points of $\widehat U$. We shall notice that the $\SO(d)$-action is not free unless $d=2$, hence quotienting out $\mathbb S$ by the $\SO(d)$ does not lead to a quotient manifold 
(actually, not even to a quotient orbifold) whenever $d\ge 3$. 

As already mentioned, central configurations give rise to simple, explicit solutions of the $n$-body problem with the property that the configuration of the particles is at any time 
similar to the initial (central) configuration. In other words, the configuration of the particles at any time is up to rotations, translations, and dilations, identical to the initial configuration.
Such explicit solutions are usually called {\sc homographic\/} or {\sc self-similar\/} solutions. In particular, every planar central configuration gives rise to a family of periodic solutions of the $n$-body problem in which each of the bodies moves on a Keplerian elliptical orbit. As a special case, the circular Keplerian orbits give rise to homographic solutions for which the configuration rigidly rotates at constant angular speed $k:=\sqrt{U(q)}$ about the center of mass.  Such periodic orbits are usually called
 {\sc (circular) relative equilibria\/}, (RE) for short, the reason being that such solutions become true equilibrium solutions in a uniformly rotating coordinate system. 

For the $n$-body problem in the physically relevant dimensions $d \le 3$ all possible homographic solutions of the $n$-body problem are those defined by central configurations. 
In higher dimension instead new interesting phenomena appear due to the greater complexity of rotations: for instance, in $\R^4$ planar and spatial central configurations can give rise to four dimensional (RE), whereas planar resp. spatial non-planar central configuration in $\R^3$ only give rise to planar (RE) resp. to homothetic collapse motions. Even more striking, the fact that in $\R^4$ it is possible to rotate in two mutually orthogonal planes with different angular speeds leads to new ways of balancing the gravitational forces with centrifugal forces in order to get new (RE) (see the notion of {\sc balanced configuration} introduced in \cite{AC98} and the further developments in \cite{AD20,AP20,AFP20}).

The nice feature of (RE) is that they become true equilibria in rotating coordinates, and as such it is natural and dynamically interesting to investigate their stability properties. 
In this paper we will focus only on 2D and 4D (RE) in $\R^4$ defined by central configurations, leaving the ``balanced case'' to future work, and perform the study 
of the stability properties of such (RE) by linearizing Newton's equation~\eqref{eq:Newton-cpt} in a uniformly rotating frame about a given (RE) and analyzing the spectral properties of the associated (linear) Hamiltonian $8n\times 8n$-matrix. This is in general a very difficult task because of several reasons: first, the dimension of the Hamiltonian matrix becomes larger as $n$ increases. Second, central configurations are not known explicitly besides in very few particular cases. Finally, the symmetries of Newton's equation~\eqref{eq:Newton-cpt} imply that the phase space of the linear Hamiltonian system at a given (RE) 
decomposes into the direct sum $E_1\oplus E_2\oplus E_3$ of three invariant symplectic subspaces, the first two of which reflect the translational resp. rotational 
invariance of the problem. The subspaces $E_1$ and $E_2$ would lead to instability if not ruled out a priori. For this reason, with slight abuse of terminology we will call a (RE) 
\begin{itemize}
\item {\sc spectrally stable\/} if the eigenvalues of the restriction of the associated Hamiltonian matrix to $E_3$ are all purely imaginary, and
\item  {\sc linearly stable\/} if in addition the associated Hamiltonian matrix restricted to $E_3$ is diagonalizable.
\end{itemize}

A long standing open problem aims at relating the linear and spectral stability properties of a (RE) with the inertia indices (as a critical point of $\widehat U$) of the central configuration generating it. 
In this respect we shall mention the following conjecture which is due to Moeckel:

\vspace{2mm}

\noindent
{\bf Conjecture (Moeckel).\/} {\em If a (RE) in the planar $n$-body problem is linearly stable, then the corresponding central configuration is a non-degenerate minimum of $\widehat U$.\/}

\vspace{2mm}

It is well-known that the converse of Moeckel's conjecture is false. For example, the (RE) corresponding to the Lagrange equilateral triangle with three equal masses is linearly unstable, even though the equilateral triangle is a non-degenerate minimum of $\widehat U$, see \cite{BJP14, BJP16} and references therein.

Even if Moeckel's conjecture still remains widely open, several progresses towards its solution have been made in the last decades. A major breakthrough has been made in 2009 by X. Hu and S. Sun \cite{HS09}. There, the authors proved that a (RE) generated by a central configuration $q$ with odd Morse index as critical point of $\widehat U$ is always linearly unstable. In particular, a non-degenerate central configuration generating a linearly stable (RE) must have even Morse index as a critical point of $\widehat U$. 
Some years later, authors in \cite{BJP14} claimed that Hu and Sun's result could be upgraded to show spectral instability of (RE) whose generating central configuration 
has even nullity and odd Morse index as a critical point of $\widehat U$. Unfortunately, the proof of their result contains a gap 
because of a missing condition about the eigenvalue 0. Namely, the statement 

\begin{center}
\textit{``If the Hamiltonian matrix $JA$, $A$ symmetric matrix, is purely imaginary, then $\iMor{A}$ is even``},
\end{center}
which is claimed in \cite[Theorem 3.11]{BJP14}, is in full generality false (see \cite{DZ21} for a counterexample). However, the proof given in \cite{BJP14} is correct under the additional 
assumption that $JA$ do not have non-trivial Jordan blocks corresponding to the eigenvalue zero. As a consequence, the main result of \cite{BJP14} still remains true in a large variety of situations,
for instance whenever the central configuration is non-degenerate. 

In the present paper we push the study of the stability properties of (RE) further by investigating the relationship between spectral and linear instability of (RE) in $\R^4$ 
and the inertia indices of the corresponding central configurations. As a byproduct, we will fix the gap in the proof of the main result of \cite{BJP14}.

As we shall see in Section 2, the aforementioned decomposition of the phase space of the linearized Hamiltonian system into symplectic invariant subspaces depends on the fact that the considered (RE) be planar or not (in particular, the dimension of $E_3$ is different in the two cases). The following theorem is the main result of the present paper. 
%
In what follows $B_3$ denotes the symmetric matrix obtained by multiplying the restriction of the Hamiltonian matrix to $E_3$ with the standard complex structure $J$. 

\begin{thm}\label{thm:main-intro}
Let $\gamma $ be a (RE) in $\R^4$. If 
\begin{itemize}
\item $\gamma$ is  non-planar, or planar but generated by a collinear central configuration, and
\[
\iMor{B_3\big|_{\ker (JB_3)^{8n-16}}} 
				-\iMor{B_3}\equiv 1 
				\quad \mod 2,
\]
or
\item $\gamma$ is planar and generated by a (planar) non-collinear central configuration, and
 \[
\iMor{B_3\big|_{\ker (JB_3)^{8n-20}}} 
				-\iMor{B_3}\equiv 1 
				\quad \mod 2,
\]
\end{itemize}
then $\gamma$  is spectrally unstable. 
\end{thm}

As in  \cite{HS09}, one of the main ingredients of the proof of Theorem~\ref{thm:main-intro} is the use of the mod 2 spectral flow for affine paths 
of Hermitian matrices. This allows us to deduce information 
about the parity of the inertia indices of $B_3$ under the assumption that the spectrum of $JB_3$ be purely imaginary. However, the proof is here considerably more 
complicated than in \cite{HS09}. Indeed, there one can use simple and classical formulas (based on the so-called crossing forms) 
for computing the spectral flow in case of regular crossings, whereas in our case one has to use refined formulas that allows to compute the spectral flow 
even in cases where the crossings are not regular (because of the possibly non-trivial Jordan blocks structure). Also, in most of the cases the starting point of the 
considered affine path will be degenerate (even after restriction to $E_3$), and this fact forces us to develop a new formula of independent interest that allows 
to compute the spectral flow in case of degenerate starting points, see Section 3.

As an application of Theorem~\ref{thm:main-intro} and the discussion about the nullity of the matrix $B_3$ resp. of the Hessian $H(q)$ (see Subsection~\ref{subsec:Morse-Bott}), 
we get the second main result of the paper relating  the linear and spectral stability properties of a (RE) with the inertia indices of the corresponding central configuration. 

\begin{thm} \label{thm:main-2-intro}
Let $q$ be a central configuration, and let $\gamma$ be the corresponding (RE). Then the following hold:
\begin{enumerate}
\item If $q$ is non-planar and Morse-Bott non-degenerate, or arbitrary and Morse-Bott degenerate with $n^0(q)$ even, then $\gamma$ is linearly unstable. 
\item Suppose that 
$$n^-\Big (B_3 \Big |_{\ker (JB_3)^{8n-8-2k}}\Big )\equiv 0 \quad \text{mod}\ 2,$$
where $k=4$ in case of a  non-planar, or planar but generated by a collinear central configuration, (RE), and $k=6$ otherwise. If 
$$n^-(q) + n^0(q) \equiv 0 \quad \text{mod}\ 2,$$
then $\gamma$ is spectrally unstable. In particular, when $q$ is Morse-Bott non-degenerate we have the following two statements:
\begin{enumerate}
\item If $q$ is non-planar and $n^-(q)$ is even, then $\gamma$ is spectrally unstable. 
\item If $q$ is planar and $n^-(q)$ is odd, then $\gamma$ is spectrally unstable.
\end{enumerate}
\end{enumerate}
\end{thm}

The first assumption in Item 2 of Theorem~\ref{thm:main-2} is satisfied for instance if 
$q$ is a Morse-Bott non-degenerate planar central configuration generating a planar (RE) $\gamma$. 

Our last result concerns the spectral and linear instability of a planar relative equilibrium in the plane. 
Such a result generalizes the main results of \cite{HS09} and \cite{BJP14} while at the same time fixing the aforementioned gap in the proof of the latter. In this case, the decomposition of the phase space is found in \cite{MS05}. For sake of simplicity, we  adopt the same notation as above.

\begin{thm} \label{thm:main-intro-planarcase}
Let $\gamma$ be a (RE) for the planar $n$-body problem. If 
 \[
\iMor{B_3\big|_{\ker (JB_3)^{4n-8}}} 
				-\iMor{B_3}\equiv 1 
				\quad \mod 2,
\]
then $\gamma$ is spectrally unstable. In particular, if $JB_3$ has no non-trivial Jordan blocks corresponding to the zero eigenvalue and   
\[
\iMor{B_3}\equiv 1 
				\quad \mod 2,
\]
then $\gamma$ is spectrally unstable. 
\end{thm}

It is worth noticing that the techniques developed in the present paper can be used for many other highly symmetric problems, for instance to prove instability results for relative equilibria
for $\SO(2p)$-invariant\footnote{For $p \ge 3$ the situation is clearly more involved because of the higher complexity of the special orthogonal group $\SO(2p)$,
and a corresponding symplectic decomposition has to be found.} simple Lagrangian systems in $\R^{2p}$ (the cases considered in the present paper corresponding to $p=1,2$)
for various interesting classes of potentials. In fact, the abstract techniques behind the main results of this paper rely only on the rotational invariance of the mechanical system, and indeed our main results offer a unified viewpoint for studying the stability of relative equilibria e.g. in the case of:
\begin{itemize}
\item[-]$\alpha$-homogeneous potentials (the Newtonian potential corresponding to the case $\alpha=1$), which are employed in different atomic models, and of the
\item[-] Lennard-Jones intermolecular potential, which is important in computational chemistry as well as in molecular modeling.
\end{itemize}

Another problem which could be dealt with with analogous methods is the so called $n$\textit{-vortex problem}. Also in this case, among all periodic orbits of particular importance are relative equilibria, 
which are rigidly rotating vortex configurations sometimes called vortex crystals. Such configurations can be characterized as critical points of the Hamiltonian function restricted to the constant angular impulse hypersurface in the phase space (topologically a pseudo-sphere whose coefficients are the circulation strengths of the vortices). From the perspective of the stability of relative equilibria, one difficulty is here 
represented by the fact that the $n$-vortex problem does not admit a Lagrangian formulation, and it is highly non-trivial to characterize the stability properties of such relative equilibria  in terms of the inertia indices of the possibly indefinite circulation matrix and of a suitable stability matrix. For further details, we refer to \cite{HPX20}.

We finish this introduction with a brief summary of the content of the paper: in Section 2 we introduce the rotating coordinates frame from (RE) in $\R^4$ and prove the symplectic decomposition of the phase space into invariant subspaces for the linearized Hamiltonian dynamics along a given (RE). In Section 3 we quickly recall the definition and basic 
properties of the spectral flow for paths of Hermitian matrices and prove a formula that allows to compute the spectral flow in case of affine paths with possibly degenerate 
starting point. In Section 4 we show how to apply the content of Section 3 to deduce information about the inertia indices of a symmetric matrix $A$ under the assumption that 
the spectrum of the Hamiltonian matrix $JA$ be purely imaginary, and in Section 5 we apply this to prove the main theorems of the paper. Finally, in the appendix we study in detail the integral 
manifolds for the Hamiltonian dynamics of the $n$-body problem.

\vspace{3mm}

\noindent \textbf{Acknowledgments.} Luca Asselle is partially supported by the DFG-grant 380257369 ``Morse theoretical methods in Hamiltonian dynamics''.  Li Wu is partially supported by the NSFC N. 12071255.

\section{Circular relative equilibria in $\R^4$}

In this section we describe the dynamical and geometrical
framework of the problem and introduce  a 
symplectic decomposition of the phase space which will enable us to rule out
the trivial Floquet multipliers produced by the integrals of motion.   
Our main reference are the beautiful lecture notes \cite{Moe14}.


\subsection{A rotating frame for circular (RE)}
\label{subsec:preliminaries-n-corpi}
Consider the Euclidean four dimensional space $\R^4$ endowed with the  Euclidean scalar product $\langle \cdot , \cdot \rangle$  and let $m_1, \ldots, m_n$, $n \ge 3$, be positive real numbers which can be thought of as  point masses of $n$ particles. 
For any position vector $q\= \trasp{(q_1, \ldots, q_n)}\in (\R^4)^n$, $q_i \in \R^4$  for every $i \in \{1, \ldots, n\}$, we can define the {\sc mass scalar product\/} and the {\sc mass norm\/}  in  $\R^{4n}$ as follows
\[
\mscal{\cdot}{\cdot}\= \langle M \cdot, \cdot \rangle \quad \textrm{ and } \quad \mnorm{\cdot}\= \langle M \cdot, \cdot \rangle^{1/2}
\]
where $M \in \Mat(4n,\R)$ is the diagonal {\sc mass matrix\/} $\diag(m_1 \Id_4, \ldots, m_n \Id_4)$,  $\Id_4$ being the $4 \times 4$ identity matrix. 

The invariance under translations of Newton's equation~\eqref{eq:Newton-cpt} implies that the center of mass of the $n$ particles has an inertial motion. For this reason there is no loss of generality 
in assuming that the center of mass lie at the origin. We define the {\sc configuration space with center of mass at the origin\/} as
 \[
 \X \= \Set{(q_1, \ldots, q_n) \in \R^{4n}| \sum_{i=1}^n m_i q_i=0}.
 \]
 It is readily seen that $\X $ is a $N$-dimensional (real) vector space, with $N\= 4(n-1)$. We define the space of {\sc collision free configurations\/} as   
\[
\widehat {\X }\=\Big \{q=(q_1, \ldots, q_n) \in \X	\ \Big | \ q_i \neq q_j \ \textrm{ for } i \neq j\Big \}= \X \setminus \Delta,
\]
where  
\[
\Delta \= \Big \{q=(q_1, \ldots, q_n) \in \R^{4n}\ |\ q_i = q_j \ \textrm{ for } i \neq j\Big \}
\]
is the  {\sc collision set\/}. In what follows, we denote the unit sphere (resp. with the collision set removed) in $\R^{4n}$ with respect to the mass scalar product   
by $\S$ (resp. by $\widehat\S$) and refer to it as the {\sc inertia ellipsoid \/} (resp. the {\sc collision free inertia ellipsoid\/}). 


As it is well known, Newton's equation~\eqref{eq:Newton-cpt} admits a Lagrangian (thus, an Hamiltonian) formulation. 
Denoting by    $T\X $ (resp. $T^*\X $) the tangent (resp. cotangent) bundle of $\X $,  the {\sc Lagrangian (function)\/} of the problem $L: T \widehat{\X } \to \R$ is defined by 
\begin{equation}\label{eq:Lagrangian}
	L(q,v)= \dfrac12 \langle M v, v\rangle + U(q),\qquad \forall (q,v)\in  T\widehat{\X } 
\end{equation}
The {\sc Hamiltonian function\/}  $H: T^*\widehat{\X }\to \R$ is given as follows 
\begin{equation}\label{eq:Hamiltonian}
	H(q,p)= pv -L(q,v)\big\vert_{v=M^{-1}p}=\dfrac12 \langle M^{-1}p, p\rangle  - U(q),\qquad \forall  (q,p)\in T^*\X, 
\end{equation}
where $p=Mv$ is the {\sc conjugate momentum\/}\footnote{Throughout in the paper we think of covectors as (column) vectors.}.
Hamilton equations of the $n$-body problem thus read
\begin{equation}\label{eq:Ham-system}
	\begin{cases}
		\dot q = \partial_p H(q,p) = M^{-1}p\\
		\dot p= -\partial_q H(q,p) = \nabla U(q) 
	\end{cases}
\end{equation}

Among all possible periodic motions of Newton's equation~\eqref{eq:Newton-cpt}, the simplest are provided by those motions in which certain configurations (among which are central configurations) are rigidly rotated (up to possibly changing the size of the configuration) with constant angular velocity about the center of mass. 

Roughly speaking, a  {\sc central configuration\/} $q$ is a configuration of the whole system at which there is a perfect  balance between the gravitational interacting forces and the position vector $q$. 
More precisely, a configuration vector $q$ is a central configuration if it solves the algebraic equation
\begin{equation}\label{eq:cc-eq-2}
	M^{-1}\nabla U(q) + \lambda q=0, \qquad \lambda= U(q)/\mnorm{q}^2.
\end{equation}

Even though Equation~\eqref{eq:cc-eq-2} is most likely impossible to solve explicitly, much can be said about its solutions. Indeed, a 
key feature of central configurations is the fact that they admit a variational characterization as critical points of the restriction $\widehat U$ of the Newtonian potential $U$ to the collision free inertia ellipsoid $\widehat{\mathbb S}$, thus allowing us e.g. to apply Morse theoretical methods in the study of central configurations. However, in doing so one has to keep in mind that central configurations 
are never isolated but rather always come in $\SO(4)$-families because of the $\SO(4)$-invariance of the potential $\widehat U$. 



\begin{dfn}\label{def:inertia-indices-cc} 
The {\sc Morse index\/} $\iMor{q}$ (resp. the {\sc Morse coindex\/} $\coiMor{q}$) of a central configuration $q$ is  the dimension of the negative (resp. positive)  spectral space of $H(q)$, the Hessian of $\widehat U$ at $q$. The non-negative integer $\nullity{q}\= \dim \ker H(q)$ is referred to as the {\sc nullity\/} of the central configuration $q$. We say that $q$ is {\sc (Morse-Bott) nondegenerate \/} if
	\[
	\nullity{q}=\dim T_q\big(\SO(4)\cdot q),
	\]
	i.e. if the nullity of $q$ is the least possible (namely, if the kernel of $H(q)$ does not contain anything which does not come from the symmetries of the problem).
	\end{dfn}

\begin{rem}
A
straightforward computation shows that 
\[
H(q):=[D^2 U(q )+ U(q)M]\big\vert_{T_q\widehat{\S }}.
\]
In the literature, in order to rule out the nullity due to the symmetries of the problem one usually works on the {\sc shape sphere\/}, namely the quotient of the collision free inertial ellipsoid by the group action.
In the planar case, namely the $n$-body problem in $\R^2$, this does not pose any additional difficulty as the $\SO(2)$-action is free, but things change in higher dimension. Indeed, as the action 
is not free (in fact, not even locally free) the quotient space is not a manifold but rather an Alexandrov space. \qed
\end{rem}

Among all periodic solutions of~\eqref{eq:Newton-cpt} which are generated by central configurations, in this paper we will be interested in the so-called circular relative equilibria, (RE) for short, which have the nice feature 
of becoming true equilibrium solutions after introducing rotating coordinates. 

Thus, let $q$ be a central configuration and let $k:=\sqrt{\lambda}$ where $\lambda$ is given by Equation~\eqref{eq:cc-eq-2}. We set
$$ \Theta (t)=\begin{bmatrix} \cos (kt) & -\sin (kt) \\ \sin (kt) & \cos (kt) \end{bmatrix}= \exp(ikt),$$
and define (real) $4 \times 4$-matrices in $\SO(4)$ by 
\begin{equation}\label{eq:R-S}
r_s(t):= \diag\big(\Theta (t), \Id\big),  \quad  \quad r_d(t):= \diag\big(\Theta (t), \Theta(-t)\big).
\end{equation}
Clearly, $r_s(t)=\exp(k_s t)$ and $r_d(t)=\exp(k_d t)$, where $k_s$ and $k_d$ are the $4\times 4$ skew-symmetric matrices defined respectively by 
\begin{equation}\label{eq:K-S}
k_s= \diag(ki, 0),  \quad  \quad k_d= \diag(ki, -ki).
\end{equation}
 We denote by $R_s(t)$ and $R_d(t)$  the $4n \times 4n$-block diagonal matrices whose (diagonal) entries
are the $4 \times 4$ matrices $r_s(t)$ and $r_d(t)$, and by $K_s$ and $K_d$ the $4n \times 4n$-block diagonal matrices whose (diagonal) entries
are the $4 \times 4$ matrices $k_s$ and $k_d$. 

Given a central configuration $q$ we define the associated circular relative equilibrium (RE) by 
\begin{equation}\label{eq:re}
q(t) \=R_d(t)\, q, \qquad  t \in \R.
\end{equation}

In rotating coordinates
\[
\begin{cases}
Q:= R_d(t) \,q\\ 
P:= R_d(t)\, p,
\end{cases}
\] 
such a (RE) becomes a true stationary solution of Hamilton's equations. The Hamiltonian function in these coordinates reads
\[
H(Q,P)=\dfrac12 \langle MP, P\rangle + \langle  K_dQ,P\rangle - U(Q),
\]
and thus Hamilton's equations are given by
\begin{equation}\label{eq:Ham-system-new}
\begin{cases}
	\dot Q= M^{-1} P + K_d Q\\
	\dot{P}= \nabla U(Q)+ K_d P
\end{cases}	
\end{equation}
By  linearizing  the Hamiltonian system \eqref{eq:Ham-system-new} we obtain the linear Hamiltonian system 
\begin{equation}
\dot z = L_d \, z = - J B_d \, z,
\label{eq:Ham-lin}
\end{equation}
where 
\begin{equation}
\label{eq:Ham-lin-2}
\quad L_d\= \begin{pmatrix}
	K_d & M^{-1}\\
	D^2 U(Q) & K_d
\end{pmatrix}, \qquad  B_d\,\= \begin{pmatrix}
 	-D^2 U(Q) & \trasp{K_d}\\
 	K_d & M^{-1}
 \end{pmatrix}.
\end{equation}
\begin{rem}
If the central configuration $q$ satisfies $q_j\in \R^2\times \{0\}\subset \R^4$ for all $j=1,...,n$, namely if $q$ is planar and contained in the $(x,y)$-plane, then the corresponding (RE) can be equivalently 
written as $q(t)=R_s(t)\, q$. We take advantage of this fact by introducing a different rotating frame for the corresponding (RE) in which we consider rotations 
only in the $(x,y)$-plane. More precisely, introducing rotating coordinates 
$$\begin{cases}
Q:= R_s(t) \,q\\ 
P:= R_s(t)\, p,
\end{cases}$$
yields after linearization again to a linear Hamiltonian system as in~\eqref{eq:Ham-lin} where the matrices
\begin{equation}\label{eq:Ham-lin-pl-2}
\quad L_s\= \begin{pmatrix}
	K_s & M^{-1}\\
	D^2 U(Q) & K_s
\end{pmatrix}, \qquad  B_s\,\= \begin{pmatrix}
 	-D^2 U(Q) & \trasp{K_s}\\
 	K_s & M^{-1}
 \end{pmatrix}
\end{equation}
replace the matrices $L_d$ and $B_d$ given by~\eqref{eq:Ham-lin-2} respectively. \qed
\end{rem}

\begin{rem}
The disadvantage of introducing rotating coordinates is that in the new coordinates the Hamiltonian is not natural (i.e. of the form kinetic plus potential energy) but rather we have an additional term appearing
(usually referred to as a \textit{Coriolis term}) which roughly speaking measures how fast the frame is rotating. However, in the rotating frame the Hamiltonian system becomes autonomous and hence, at least in principle, the monodromy matrix (thus the Floquet multipliers) can be explicitly computed. \qed 
\end{rem}


\subsection{A symplectic decomposition into invariant subspaces}\label{subsec:decomposition}

The goal of this section is to provide a symplectic decomposition of the phase space into subspaces which are invariant by the phase flow of the linear Hamiltonian system~\eqref{eq:Ham-lin}.
In such a way we will be able to rule out the uninteresting part of the dynamics corresponding to the symmetries of the $n$-body problem.
More precisely, we will construct a decomposition  $T^*(\R^{4n})\cong \R^{8n}$ as the direct sum of three invariant linear symplectic subspaces $E_1, E_2$ and $E_3$: The subspace $E_1$ will reflect the translational invariance of the problem, whereas the subspace $E_2$ the invariance under rotations. Finally, the subspace $E_3$ will be obtained as the symplectic orthogonal complement of $E_1\oplus E_2$.

Such a decomposition is inspired by Meyer and Schmidt's work \cite{MS05}, but unlike in the Meyer and Schmidt's case will depend on the fact that the (RE) be planar or not as well as on the 
fact that the central configuration generating the (RE) be planar or not. 

Without loss of generality we suppose hereafter that 
$$\sum_{i=1}^n m_i = 1.$$

\noindent \textbf{The subspace $E_1$.} We start to define the linear subspace $E_1 \subset \R^{8n}$ corresponding to the conservation of the center of mass and linear momentum and thus reflecting the translational invariance of the $n$-body problem. 
The construction of $E_1$ is independent on the (RE) being planar or not as well as on the generating central configuration be planar or not. In what follows the matrices $L_d$ and $L_s$ defined respectively in~\eqref{eq:Ham-lin-2} and in~\eqref{eq:Ham-lin-pl-2} (according to the fact that the (RE) be planar or not) will be denoted with $L_*$ whenever we don't need to distinguish them. 
We thus set $E_1\subset  \R^{8n}$ to be the $8$-dimensional subspace given by
 \[
E_1\=\Big \{(v, \ldots, v, m_1 w, \ldots, m_n w)\in T^*\R^{4n}\ \Big |\   v,w \in \R^4\Big \}
\]
We denote by $\{e_l\}$ the standard basis of $\R^4$ and set $\ve_l=(e_l, \ldots, e_l)\in \R^{4n}$, $l=1,...,4$. A basis of $E_1$ is then given by
\begin{equation}
 u_l=\trasp{(e_l, \ldots , e_l, 0, \ldots ,0)}= \trasp{(\ve_l, \zero)} 
\quad \textrm{ and } \quad 
v_l =\trasp{(0,\ldots ,0, m_1 e_l, \ldots, m_n e_l)}=\trasp{(\zero , M \ve_l)}
\end{equation}

\begin{lem}\label{thm:lemma1}
$E_1$ is an $L_*$-invariant symplectic subspace of $(\R^{8n}, \omega_0)$. Furthermore, the restriction  $L_1^d$ of $L_d$ to $E_1$ is represented with respect to the basis $\{u_1,u_2,v_1,v_2,u_3,u_4,v_3,v_4\}$ by
\begin{equation}\label{eq:L1}
L_1^d=\displaystyle L_d|_{E_1}
 = \begin{pmatrix}
 	k i & \Id & 0 & 0 \\
 	0 & -ki & 0 & 0\\
 	0 & 0 & ki& \Id\\
 	0 & 0 & 0 & -ki
 \end{pmatrix},
\end{equation}
whereas the restriction $L_1^s$ of $L_s$ to $E_1$ is represented by 
\begin{equation}\label{eq:L1s}
L_1^s:\displaystyle L_s|_{E_1}
 = \begin{pmatrix}
 	k i & \Id & 0 & 0 \\
 	0 & -ki & 0 & 0\\
 	0 & 0 & 0& \Id\\
 	0 & 0 & 0 & 0
 \end{pmatrix}.
 \end{equation}
\end{lem}
\begin{proof}
By a direct computation, we get that  
\begin{align}
\omega_0(u_l, u_m) &=\omega_0(v_l,v_m)=0, \quad \forall l,m,\\ 
\omega_0(u_l, v_l) &=\langle J u_l, v_l\rangle= \langle M \ve_l, \ve_l\rangle=1, \quad \forall l, \\
\omega_0(u_l, v_m) &=0, \quad \forall l \neq m, 
\end{align}
thus showing that $\{u_1,...,u_4,v_1,...,v_4\}$ is a symplectic basis of $E_1$.

Now, since
\begin{equation}
L_d=\begin{pmatrix}
K_d & M^{-1}\\
D^2 U(q ) & K_d	
\end{pmatrix}, \qquad  K_d= k \diag \big ((i,-i),...,(i,-i)\big ) = k \diag(\vi, \ldots,  \vi)
\end{equation}
we easily compute
\begin{equation}\label{eq:restrizione-L-E1} 
\begin{aligned}
L_d\begin{pmatrix}\ve_l \\  0 
\end{pmatrix} &= \begin{pmatrix}
K_d \ve_l \\  D^2 U(q ) \ve_l
\end{pmatrix}= \begin{pmatrix}
K_d \ve_l \\  0
\end{pmatrix}
 \\
L_d\begin{pmatrix}
0\\  M \ve_l
\end{pmatrix} &= \begin{pmatrix}
	 \ve_l \\ K_d M \ve_l 
\end{pmatrix}= \begin{pmatrix}
 \ve_l \\ 0 
\end{pmatrix}+ \begin{pmatrix} 0\\  M K_d \ve_l 
\end{pmatrix} = u_l +  \begin{pmatrix} 0\\  M K_d \ve_l 
\end{pmatrix}
\end{aligned}
\end{equation}
where $D^2 U(q ) \ve_l=0$ follows by differentiating and evaluating at $t=0$ the identity 
$$\nabla U(q+t\ve_l)=\nabla U(q).$$
The $L_d$-invariance of $E_1$ follows observing that $K_d \ve_l$ is again a vector of the form $\pm k \ve_j$. 
The $L_s$-invariance of $E_1$ is proven in a completely analogous fashion, and to obtain the corresponding matrix representation one just needs to notice that $K_s \ve_l=0$ for $l=3,4$.
\end{proof}

\begin{rem}
As we readily see from the lemma above, $L_1^*$ has a non-trivial Jordan block structure (roughly speaking, such Jordan blocks are associated with the drift that one gets when allowing non-zero total momentum)
and this leads to linear instability. The idea is therefore to exclude the subspace $E_1$ when looking at the stability properties of (RE), thus ruling out an uninteresting part of the dynamics. \qed
\end{rem}


\vspace{3mm}

\noindent \textbf{The subspace $E_2$.} The second invariant subspace $E_2$ reflects the rotational and scaling invariance of the $n$-body problem, but unlike the subspace $E_1$ it will strongly depend on the fact that 
the (RE) be planar or not. 
Before defining $E_2$ we shall recall some basic facts about the special orthogonal group $\SO(4)$. 
%
%
%
From an algebraic viewpoint $\SO(4)$ is a non-commutative compact $6$-dimensional Lie group which is commonly identified with the group 
of real $4\times 4$ orthogonal matrices with determinant $1$, namely with the group of rotations in $\R^4$. In $\R^4$ we have two types of rotations: 
\begin{itemize}
	\item[(s)] {\sc simple rotations}, which leave a linear plane $\Pi$ fixed. If we identify $\Pi$ with the $\{0\}\times \R^2$-plane, then a simple rotation of angle $\varphi$, takes the form $\diag(\Phi, \Id)$ where 
	$$\Phi=\begin{bmatrix}
	\cos \varphi & -\sin \varphi\\
	\sin \varphi & \cos \varphi
\end{bmatrix}.$$   
	\item[(d)] {\sc double rotations}, which fix only the origin. In this case  there exists a pair of orthogonal planes $\Pi_1$ and $\Pi_2$ each of which is invariant. Hence, such a double rotation produces an usual planar rotation on each of the planes $\Pi_1$ and $\Pi_2$. If the rotation angles of such a double rotation on $\Pi_1$ and $\Pi_2$ coincide, then there are
infinitely many invariant planes. In this case we say that the double rotation is {\sc isoclinic\/}.
\end{itemize}
We have two types of isoclinic rotations in $\R^4$ which are specified by the same angle $\varphi$, namely: 
\begin{itemize}
	\item {\sc left-isoclinic rotations\/}, if the signs of the rotations are opposite. Such a class of rotations contains $\diag(\Phi,- \Phi)$ and $\diag(-\Phi,\Phi)$,
	\item {\sc right-isoclinic rotations\/}, if the signs of the rotations are equal. Such a class of rotations contains $\diag(\Phi, \Phi)$ and $\diag(-\Phi,-\Phi)$.
\end{itemize} 
The well-known {\sc isoclinic decomposition\/}, also known as {\sc  Van Elfrinkho's formula\/}, states that, up to a central inversion, every 
4D rotation is the product of a left-isoclinic and a right-isoclinic rotation. 

Left- (resp right-)isoclinic rotations form a non-commutative subgroup $S^3_L$ (resp. $S^3_R$) of $\SO(4)$ which is isomorphic to the multiplicative group $\S^3$ of unit quaternions. It can be shown that $S^3_L \times S^3_R$ is the universal covering of $\SO(4)$, and that $S^3_L$ and $S^3_R$ are normal subgroups of $\SO(4)$. 

Moreover, left- and right-isoclinic rotations can be described in terms of quaternions as we now recall. 
 Quaternions are generally represented in the form 
\[
Q= a+ b \vi + c \vj+ d \vk
\]
where $a,b,c$ and $ d$ are real numbers and $\vi, \vj$ and $\vk$ are the {\sc quaternion units\/}. We have
\[
\vi^2=\vj^2=\vk^2=\vi\vj\vk=-1.
\]
The quaternion units can be represented as $2 \times 2$-complex matrices as follows 
\begin{equation}
\vi= \begin{pmatrix}
	i & 0 \\ 0 & -i
\end{pmatrix}, \qquad \vj= \begin{pmatrix}
	0 & 1 \\ -1 & 0
\end{pmatrix},\qquad 
\vk= \begin{pmatrix}
	0& i  \\ i&0
\end{pmatrix}
\end{equation}
 In real notation, $\vi, \vj, \vk$ can be written, respectively, as the following real matrices
\begin{equation}
\vi= \begin{pmatrix}
	0&-1 & 0 &0\\  1 & 0& 0 & 0 \\
	0 & 0 & 0 & 1 \\ 0 & 0& -1 & 0
\end{pmatrix}, \qquad 
\vj= \begin{pmatrix}
	0&0 & 1 &0\\  0 & 0& 0 & 1 \\
	 -1 & 0 &0 & 0 \\  0 & -1 & 0&0 \
\end{pmatrix}, \qquad
\vk= \begin{pmatrix}
	0&0 & 0 &-1\\  0 & 0& 1 & 0 \\
	0 & -1 & 0 & 0 \\ 1 & 0& 0 & 0
\end{pmatrix}
\end{equation}
A point $(x,y,z,t)\in \R^4$ can be represented by the quaternion $P=x+ y\vi+z\vj+t\vk$, and it is straightforward to check that left-isoclinic (resp. right-isoclinic) rotations correspond to the left (resp. right) multiplication by a unit quaternion. 

We can now finally proceed with the definition of the subspace $E_2$. 
We start considering the case in which the (RE) is non-planar (independently of the fact that the central configuration generating it be planar or not). Thus, let $q$ be a central configuration
generating a non-planar (RE). We define $8$ linearly independent vectors as follows 
\begin{align*}
 z_1& = \trasp{(q, \zero)}, \qquad 
 z_2=\trasp{(\vi\, q , \zero)}, \qquad 
 z_3= \trasp{(\vj \,q ,\zero)}, \qquad 
 z_4= \trasp{(\vk \,q,   \zero)}\qquad \\
 w_1& =\trasp{(\zero, Mq  )}, \qquad 
 w_2= \trasp{(\zero, \vi\, Mq )}, \qquad 
 w_3= \trasp{(\zero, \vj\, M q)}, \qquad 
 w_4= \trasp{(\zero, \vk\, Mq)},  
\end{align*}
where with slight abuse of notation $\vi \cdot $, $\vj \cdot$, $\vk \cdot$ denotes the diagonal left-multiplication in $\R^{4n}$ of the quaternion $\vi,\vj,\vk$ respectively, and set 
\[
E_2^{d}=\mathrm{Span}\, \big \{ z_1,z_2, w_1, w_2,  z_3, z_4, w_3, w_4\big\}.
\]

\begin{lem}\label{thm:lemma2}
Let $q$ be a central configuration generating a non-planar (RE). Then $E_2^{d}$ is an $L_d$-invariant symplectic subspace of  $(\R^{8n}, \omega_0)$. Furthermore, the restriction of $L_d$ to $E_2^d$ is represented by the $8\times 8$-matrix 
\begin{equation}\label{eq:E2n}
  L_2^d\=L|_{E_2^d}
 = \begin{pmatrix}
 	k i & \Id & 0 & 0 \\
 	A & ki& 0 & 0\\
 	0 & 0 & ki& \Id\\
 	0 & 0 & B & ki
 \end{pmatrix},\end{equation}
 where 
 $$A=\begin{pmatrix}
	2k^2 & 0\\ 0 & -k^2
\end{pmatrix}, \quad
B=\begin{pmatrix}
	-k^2 & 0\\ 0 & -k^2
\end{pmatrix}.
$$
\end{lem}

\begin{proof}
It is straightforward to check that $E_s^d$ is a symplectic subspace.
To show that $E_2^d$ is $L_d$-invariant 
we start computing
$$L_d z_1=\binom{k \vi q}{D^2 U(q ) [q]} =k \binom{\vi q}{\zero} + \binom{\zero}{D^2 U(q )[q]} = -k z_2+2k^2 w_1,$$
where we used the fact that $\nabla U$ is $(-2)$-homogeneous and the fact that $q$ is a central configuration to infer that $D^2U(q)[q]=-2\nabla U(q) = 2k^2 Mq$. Similarly, we compute 
$$L_d z_2=\binom{k \vi\vi  q}{D^2 U(q )[\vi q ]}= -k \binom{q}{\zero}+ \binom{\zero}{-k^2 M\vi q}= -k z_1-k^2w_2,$$
where to obtain the last equality we have differentiated the identity\footnote{This follows from the fact that $\left (\begin{matrix} e^{it} & 0 \\ 0 & e^{-it}\end{matrix}\right ) \cdot q$ is a central configuration for every $t\in \R$.}
$$\nabla U \left ( \left (\begin{matrix} e^{it} & 0 \\ 0 & e^{-it}\end{matrix}\right ) \cdot q \right ) = - k^2 M  \left (\begin{matrix} e^{it} & 0 \\ 0 & e^{-it}\end{matrix}\right ) \cdot q$$ 
and evaluated it at $t=0$. In a similar way we compute

\begin{equation}
\begin{aligned}
&L_d  z_3=\begin{pmatrix}
k \vi\vj  q \\
D^2 U(q )[\vj q ] 
\end{pmatrix}= k \begin{pmatrix}
\vk q \\
\zero
\end{pmatrix}-k^2 \begin{pmatrix}
\zero \\
M\vj  q  
\end{pmatrix}= k z_4-k^2 w_3\\
&L_d z_4=\begin{pmatrix}
k \vi\vk  q \\
D^2 U(q )[\vk q ] 
\end{pmatrix}= - k z_3-k^2 w_4\\
&L_d  w_1=
\begin{pmatrix}
 q \\
kM\vi q   
\end{pmatrix}= z_1+k w_2\\ 
&L_d w_2=
\begin{pmatrix}
 \vi q \\
kM\vi^2 q   
\end{pmatrix}= z_2-k w_1\\
&L_d w_3=\begin{pmatrix}
 \vj q \\
kM\vi\vj q   
\end{pmatrix}= z_3+k w_4\\
&L_d w_4=\begin{pmatrix}
 \vk q \\
kM\vi\vk q   
\end{pmatrix}= z_4-k w_3\\
\end{aligned}
\end{equation}
This completes the proof. 
\end{proof}

\begin{rem} An equivalent definition of $E_2^d$ is given as follows:
$$E_2^d := \Big \{ \big ( \alpha \cdot (\mathfrak A q,...,\mathfrak A q), \beta \cdot(\mathfrak B Mq, ... \mathfrak B Mq) \ \Big |\ \alpha,\beta \in \R, \ \mathfrak A,\mathfrak B \in SU(2)\Big \}.$$
This should be compared with the definition given in \cite[Page 259]{MS05} for the planar $n$-body problem. In fact, if one wanted to naively generalize Meyer and Schmidt's construction to $\R^4$ then 
one would be tempted to define the subspace $E_2^d$ as follows 
$$\Big \{ \big ( \alpha \cdot (\mathfrak A q,...,\mathfrak A q), \beta \cdot(\mathfrak B Mq, ... \mathfrak B Mq) \ \Big |\ \alpha,\beta \in \R, \ \mathfrak A,\mathfrak B \in SO(4)\Big \}.$$
However, this unfortunately does not lead to an $L_d$-invariant subspace. Consider for instance the vector 
$$z_5:=\trasp{ (q \vi,\zero)},$$
where as usual one identifies right-isoclinic rotations in $\SO(4)$ with the right-multiplication by quaternions. Computing
\begin{align*}
L_d z_5 &= \binom{K_d q \vi}{\zero} + \binom{\zero}{D^2U(q)[q \vi]} = k \binom{\vi q \vi}{\zero} - k^2 \binom{\zero}{M q \vi},\\
L_d \binom{\vi q \vi}{\zero} &=-k \binom{q \vi}{\zero} + \binom{\zero}{D^2U(q)[\vi q \vi]},
\end{align*}
we see that $D^2U(q)[\vi q \vi]$ cannot be computed. Indeed, one easily checks that $\vi q \vi$ cannot be written as $\mathfrak a q$ for some $\mathfrak a \in \mathfrak s \mathfrak o(4)$ (we leave 
the easy proof of this fact to the reader). 

In the planar case, Meyer and Schmidt's definition still gives rise to an invariant subspace since $\SO(2) =SU(1)$. Another reason why one cannot use $\SO(4)$ instead of $SU(2)$ in our 
setting is the following: if this was the case, then the space $E_2^d$ would have dimension 14, and this would not allow us to find the symplectic change of coordinates given in Proposition~\ref{thm:new-coordinates}
since 14 is not an integer multiple of 4.
 \qed
\end{rem}

We now move on to investigate the case of a planar (RE) $\gamma$. In this case, the central configuration generating $\gamma$ might be collinear or planar non-collinear. 
We start by considering the latter case. Thus, let $q=\trasp{(x,y,0,0)}$ be a non-collinear central configuration contained in the plane $\R^2\times \{0\}\subset \R^4$. 
The corresponding (RE) $\gamma$ is pointwise defined by $\gamma(t)= R_s(t) q$ where $R_s(t)$ is as defined in Section~\ref{subsec:preliminaries-n-corpi}, and the Hamiltonian matrix $L_s$ is given by 
\[
L= \begin{bmatrix}
	K_s &  M^{-1}\\
	D^2 U(q) & K_s
\end{bmatrix}
\]
where 
$$ K_s=k \diag \big ((i,0),...,(i,0)\big ).$$
For $1 \le l < r \le 4$, we set $k_{lr} \in \Mat(4 \times 4, \R)$ by 
\[
(k_{lr})_{ij}= \begin{cases}
 -1 & (i,j)=(l,r)\\
 1 & (i,j)= (r,l)\\
 0 & \textrm{ otherwise} 	
 \end{cases}
\]
By using this notation, the six simple rotations of $\SO(4)$ are the following $k_{12}, k_{13}, k_{14}, k_{23}, k_{24}, k_{34}$, and it is straightforward to check that 
\begin{align}
& \vi= k_{12}-k_{34} && \vj= -k_{13}-k_{34} && \vk= k_{14}- k_{23}	
\end{align}
We denote with $K_{lr}$ the $4n\times 4n$-dimensional block diagonal matrix defined by 
\[
K_{lr}= \diag\big(k_{lr}, \ldots, k_{lr}\big).
\]
Under the notation above (and noticing that $K_{34} q=q$) we set
\begin{align}
& z_1= \trasp{(q,\zero)},&& z_2= \trasp{(K_{12}q,\zero)}, && z_3= \trasp{(K_{13}q,\zero)}, &&z_4= \trasp{(K_{14}q,\zero)},  \\	
& z_5=\trasp{(K_{23}q,\zero)}, &&  z_6=\trasp{(K_{24}q,\zero)},\\
& w_1= \trasp{(\zero,Mq)},&& w_2= \trasp{(\zero, MK_{12}q)}, && w_3= \trasp{(\zero,MK_{13}q)}, &&w_4= \trasp{(\zero, MK_{14}q)}, \\
& w_5= \trasp{(\zero,MK_{23}q)},&& w_6= \trasp{(\zero, MK_{24}q)}.
\end{align}
It is straightforward to check that, since $q$ is not collinear, the vectors $z_1,...,z_6,w_1,...,w_6$ are linearly independent in $\R^{8n}$. We can thus define the $12$-dimensional linear subspace
\begin{equation}\label{eq:E2p}
E_2^s:=\mathrm{Span}\big\{ z_1,z_2, w_1, w_2, z_3, z_4, w_3,w_4, z_5, z_6, w_5,  w_6\big \}.
\end{equation}

\begin{lem}\label{thm:lemma3}
Let $\gamma$ be a planar (RE) generated by a non collinear central configuration $q$. Then the  subspace $E_2^s$ is a $L_s$-invariant symplectic subspace of $(\R^{8n},\omega_0)$, and the restriction of $L_s$ to $E_2^s$ is represented by
\begin{equation}\label{eq:E2p}
 L_2^s\=L|_{E_2^s}
 = \begin{pmatrix}
 	k J & \Id & 0 & 0 & 0 & 0 \\
 	A & kJ & 0 & 0& 0 & 0 \\
 	0 & 0 & 0& \Id& 0 & 0 \\
 	0 & 0 & B & 0& 0 & 0 \\
 	0 & 0 & 0 & 0 & 0 & \Id\\
 	0 & 0 & 0 & 0 & B & 0
 \end{pmatrix},\end{equation}
where $A$ and $B$ are as in Lemma~\ref{thm:lemma2}.
\end{lem}
\begin{proof}
	We start observing that 
	\begin{align*}
		\omega_0(z_i, z_j)& =\omega_0(w_i, w_j)=0 \quad \forall\, i, j,\\
		\omega_0(z_1, w_1)& = \langle \binom{\zero}{q},\binom{\zero}{Mq}\rangle =1,\\
		\omega_0(z_2, w_2)&= \langle \binom{\zero}{K_{12}q},\binom{\zero}{MK_{12}q}\rangle= \sum_{i=1}^n m_i |q_i|^2=1,\\
		\omega_0(z_3,w_3) &= \langle \binom{\zero}{K_{13}q},\binom{\zero}{MK_{13}q}\rangle = \sum_{i=1}^n m_i x_i^2 >0,\\
		\omega_0(z_4,w_4) &= \langle \binom{\zero}{K_{14}q},\binom{\zero}{MK_{14}q}\rangle = \sum_{i=1}^n m_i x_i^2 >0,\\
		\omega_0(z_5,w_5) &= \langle \binom{\zero}{K_{23}q},\binom{\zero}{MK_{23}q}\rangle = \sum_{i=1}^n m_i y_i^2 >0,\\
		\omega_0(z_6,w_6) &= \langle \binom{\zero}{K_{24}q},\binom{\zero}{MK_{24}q}\rangle = \sum_{i=1}^n m_i y_i^2 >0,
	\end{align*}
where we used the fact that $q$ is not collinear and hence at least one of the $x_i$ respectively $y_i$ does not vanish. 
This shows that the restriction of $\omega_0$ to $E_2^s$ is non-degenerate and hence $E_2^s$ is a symplectic subspace of $(\R^{8n},\omega_0)$. Notice however that, since $\omega_0(z_i,w_j)\neq 0$ in 
general, the basis $\{z_1,...,z_6,w_1,...,w_6\}$ is not a symplectic basis of $E_2^s$, not even after renormalization. 

By a direct computation, noticing that $K_s=k K_{12}$ and that 
$$D^2U(q)[K_{ij}q]=-k^2 K_{ij}q, \quad \forall (i,j)\neq (3,4),$$ 
we get 
\begin{align}
L_sz_1& = \binom{K_s q}{\zero} + \binom{\zero}{D^2U(q)[q]}=  k \binom{K_{12} q}{\zero} + 2k^2 \binom{\zero}{M q} = k z_2 + 2k^2 w_1,\\
L_sz_2&= k \binom{K_{12} K_{12} q}{\zero} + \binom{\zero}{D^2 U(q)[K_{12} q]} = - k z_1 - k^2 w_2,\\
L_sw_1&=  z_1+kw_2,\\
 L_sw_2&= z_2-kw_1,\\
  L_sz_j&= -k^2 w_j, \ \forall j \ge 3,\\ 
  L_s w_j&=  z_j, \ \forall j \ge 3. 
\end{align}
This shows at once that $E_2^s$ is $L_s$-invariant and that the restriction of $L_s$ to $E_2^s$ has the desired matrix representation with respect to the basis
$\{z_1,z_2,w_1,w_2,z_3,z_4,w_3,w_4,z_5,z_6,w_5,w_6\}.$ 
\end{proof}

In case the planar (RE) $\gamma$ is generated by a collinear central configuration it is easy to check that the subspace $E_2^s$ has dimension 8 and coincides with the subspace $E_2^d$ defined 
in the case of non-planar (RE). Moreover, one can use the same basis as for $E_2^d$. We leave the easy proof of this fact to the reader. 

\begin{rem}
As we have seen above, the invariant subspace $E_2$ depends both on the (RE) and on the central configuration generating it. Nevertheless,
$L_2^*$ leads in either case to linear instability (exactly as $L_1^*$) because of its non-trivial Jordan blocks structure. \qed
\end{rem}

%


\vspace{3mm}

\noindent \textbf{The invariant subspace $E_3$.} We define the linear subspace $E_3^d$ (resp. $E_3^s$) to be the symplectic orthogonal complement of $E_1 \oplus E_2^d$ (resp. $E_1 \oplus E_2^s$), and denote by $E_3^*$ either one of the subspaces $E_3^d, E_3^s$ whenever there is no need to distinguish between them. Clearly, $E_3^d$ (resp. $E_3^s$) is itself invariant, being the symplectic orthogonal complement of an invariant subspace.
By counting dimensions, we get that 
$$
\dim E_3^d= 8n-16,  \qquad \dim E_3^s= \left \{ \begin{array}{l} 8n-20 \quad \text{if} \ q \ \text{non-collinear},\\ 8n-16 \quad \text{if}\ q \ \text{collinear}.\end{array}\right.
$$

Summarizing the contents of the previous subsections, we have obtained a decomposition of the phase space into symplectic subspaces which are invariant for the linear Hamiltonian system. 

\begin{prop}\label{thm:invariant-decomposition}
Let $q$ be a central configuration generating a non-planar (RE) $\gamma$. Then
\[
\R^{8n}= E_1 \oplus E_2^d \oplus E_3^d
\]
is a decomposition of the phase space in symplectic subspaces which are invariant under the linear Hamiltonian flow~\eqref{eq:Ham-lin}.
A similar statement holds for planar (RE) replacing $E_2^d,E_3^d$ with $E_2^s,E_3^s$ respectively. \qed
\end{prop}

The next result provides an explicit symplectic change of coordinates which allows to split the original Hamiltonian system into three Hamiltonian subsystems, each of which is defined on the corresponding symplectic subspace appearing in the above decomposition. In what follows we have $k=4$ or $6$, where $k=4$ holds in case of a non-planar (RE) or of a planar (RE) generated by a collinear central configuration, and $k=6$ 
otherwise.

\begin{lem}\label{thm:new-coordinates}
Let $X=(g,z,w) \in \R^4 \times \R^k \times \R^{8n-4-k}$ and $Y=(G,Z,W) \in \R^4 \times \R^k \times \R^{8n-4-k}$. There exists a linear symplectic transformation in $\R^{8n}$  of the form 
\[
\begin{cases}
	Q= AX\\
	P= \traspinv{A} Y
\end{cases}
\]
where $A$ is an $M$-orthogonal matrix commuting with $K_*$ (here, $*=d$ in case of a non-planar (RE), and $*=s$ otherwise).
Furthermore  $(g,G)$, $(z,Z)$ and $(w,W)$ are symplectic coordinates of $E_1\oplus E_2^*\oplus E_3^*$.
\end{lem}

\begin{proof}
We prove the claim in case of a non-planar (RE) (in particular, $k=4$) leaving the other cases to the reader. We first construct the first eight columns of the $4n\times 4n$-matrix $A$, and then construct the remaining columns by means of the Gram-Schmidt orthonormalization method with respect to the $M$-scalar product. 

Thus consider the following 8 vectors in $\R^{4n}$:
$$r_l := \ve_l, \quad l=1,...,4,\quad r_5 := q, \quad r_6 := \vi q, \quad r_7 := -\vj q, \quad r_8 := \vk q,$$
where as usual $\vi\cdot,\vj\cdot,\vk\cdot$ denote with slight abuse of notation the diagonal left-multiplication in $\R^{4n}$ with the quaternions $\vi,\vj,\vk$ respectively. 
We shall stress the fact that there is a minus sign in front of $\vj$. In fact, this is needed in order to have that $A$ commutes with $K_d$, which we recall is nothing else but $k\, \text{diag}(\vi,...,\vi)$.

Since we have normalized the masses in such a way that their sum is equal to 1, the $M$-norm of $r_l$ is equal to 1 for every $l=1,...,4$. Also, the $M$-norm of $r_5,...,r_8$ is also equal to 1 since $q$ is by assumption a normalized central configuration. 
It is straightforward to check that  $r_5, r_6, r_7 , r_8$ are $M$-orthogonal to each other (this is a direct consequence of the multiplicative table of the imaginary units and the fact that in each $4 \times 4 $-block the matrix $M$ is a scalar multiple of the identity). Finally, we observe that $r_5, r_6, r_7, r_8$ are $M$ orthogonal to the first four since the center of mass lies at the origin by assumption.

Arguing recursively we construct the remaining $4n-8$ vectors by quadruples. We set 
$$V_1:= \mathrm{Span}\{r_1, \ldots r_8\}$$
and choose $ s_9\notin V_1$. Denoting with $\pi_{V_1}^M:\R^{4n}\to \R^{4n}$ the $M$-orthogonal projection onto $V_1$, we have that 
$$r_9 := \frac{\tilde s_9}{|\tilde s_9|_M}, \qquad \tilde s_9 := s_9 - \pi_{V_1}^M s_9,$$
is orthogonal to $V_1$ and has $M$-norm 1. We finally set 
$$r_{10}:= \vi r_9, \quad r_{11} := -\vj r_9,\quad r_{12}:= \vk r_9.$$
Arguing as before, we see that the $M$-normalized vectors $r_9,...,r_{12}$ are pairwise $M$-orthogonal to each other as well as $M$-orthogonal to $r_1,...,r_8$. 
The remaining columns of the matrix $A$ are constructed analogously. One easily checks that $A$ commutes with $K_d$. Indeed, writing $A=(A_{lm})_{l,m=1,...,n}$, where each $A_{lm}$ is a $4\times 4$-block,
it is straightforward to check that the condition $AK_d=K_d A$ is equivalent to 
$$\vi \cdot A_{ij} = A_{ij} \cdot \vi, \quad \forall i,j,$$
where here we identify $\vi$ with the corresponding real $4\times 4$-matrix.
Now, one readily computes for the $4\times 4$-block (being the computations for the remaining blocks identical)
$$A_{21}= (q_1,\vi q_1 , -\vj q_1,\vk q_1) = \left ( \begin{matrix} x_1 & - y_1 & -z_1 & - w_1 \\ y_1 & x_1 & -w_1  & z_1 \\ z_1 & w_1 & x_1 & -y_1  \\ w_1 & -z_1 & y_1 & x_1 \end{matrix}\right ),$$
where $q_1=\trasp{(x_1,y_1,z_1,w_1)}$ is the first vector of the central configuration $q$, 
$$\vi \cdot A_{21} = \left ( \begin{matrix} -y_1 & -x_1 & w_1 & - z_1 \\ x_1 & -y_1 & -z_1  & -w_1 \\ w_1 & -z_1 & y_1 & x_1  \\-z_1 & -w_1 & -x_1 & y_1 \end{matrix}\right ) = A_{21}\cdot \vi.$$
The matrix $A=(r_1,...,r_{4n})$ is by construction $M$-orthogonal, that is we have $\trasp A MA =I$, which can be equivalently rewritten as $\traspinv{A}=MA$. 
In particular, the first eight columns of $\traspinv{A}$ are given by 
$$t_l = M \ve_l, \quad l=1,...,4, \quad t_5 =Mq,\quad t_6 = M\vi q, \quad t_7 = - M\vj q,\quad t_8=M\vk q.$$

This concludes the  proof of the first claim.   The second claim follows as well noticing that 
$E_1$ is generated by the vectors $\trasp{(r_l,\zero)}, \trasp{(\zero,t_l)}$, $l=1,...,4$, and
$E_2$ is generated by the vectors  $\trasp{(r_l,\zero)}, \trasp{(\zero,t_l)}$, $l=5,...,8$.
\end{proof}

The next result is a straightforward consequence of Lemma~\ref{thm:new-coordinates} and provides the structure of the matrix $L$  in the new coordinates system.

\begin{thm}\label{thm:structure-L}
Let $\gamma$ be a non-planar (RE). Then, after the symplectic change of coordinates provided in Lemma~\ref{thm:new-coordinates}, the Hamiltonian matrix $L_d$ given by~\eqref{eq:Ham-lin} takes the form
\[
L_d=\begin{pmatrix}
L_1& 0 & 0\\	
0 & L_2^d & 0 \\
0& 0 & L_3^d
\end{pmatrix}
\]
where $L_1$ is given by~\eqref{eq:L1}, $L_2^d$ is given by~\eqref{eq:E2n}, and
\begin{equation}\label{eq:L3*}
\displaystyle  L_3^d\=L|_{E_3^d}= \begin{pmatrix}
	K_d & \Id    \\
	 \mathcal D & K_d
\end{pmatrix}, \qquad \mathcal D= \big[\trasp{A} D^2 U(Q) A\big]\Big|_{\R^{4n-8}},
\end{equation}
where with slight abuse of notation we denote the $(4n-8)\times (4n-8)$-matrix $k \diag (\vi,...,\vi)$ again with $K_d$. \qed
\end{thm}

\begin{rem}
An analogous statement holds also in case of a planar (RE) replacing $L_d$, $K_d$ with $L_s$, $K_s$ respectively, and in case the central configuration generating it is planar replacing $4n-8$ with $4n-10$. Notice 
however that in this latter case the representation of $L_2^s$ in the new coordinates is not the one given by the matrix in~\eqref{thm:lemma3}, since the considered basis was not symplectic. \qed
\end{rem}



\subsection{More on $L_3^*$ and the corresponding symmetric matrix $B_3^*=JL_3^*$}
\label{subsec:Morse-Bott}

In the last section we provided a symplectic change of coordinates, inspired by Meyer and Schmidt's one \cite{MS05}, 
that allows to rule out the uninteresting part of the linearized dynamics coming from the symmetries of the $n$-body problem. 
However, as we shall see below, unlike in the planar case one is in general not able to rule out the nullity of 
the Hessian of $\widehat U$ at the central configuration $q$ completely, not even after such a symplectic decomposition and not even under the assumption that 
the considered central configuration be Morse-Bott non-degenerate. 

More precisely, consider the Hamiltonian matrix 
$$L_3^* = \left (\begin{matrix} K_* & \Id \\ \mathcal D & K_* \end{matrix} \right ),$$
where $*$ stands either for $d$ or $s$, according to the (RE) being non-planar or planar, and the associated symmetric matrix (notice that $\trasp{K_*}=-K_*$)
$$B_3^*:=JL_3^* = \left (\begin{matrix} - \mathcal D & -K_*\\ K_* & \Id\end{matrix}\right ).$$ 
The identity
\begin{equation*}
\begin{pmatrix}
	\Id & K_* \\
	0 & \Id 
\end{pmatrix}\begin{pmatrix}
	-\mathcal D & -K_*\\
	K_* & \Id
\end{pmatrix}\begin{pmatrix}
	\Id & 0\\
	-K_* & \Id
\end{pmatrix}=
\begin{pmatrix}
-\big[\mathcal D + U(q)\big]	 & 0 \\
 0 & \Id 
\end{pmatrix}
\end{equation*}
shows that $B_3^*$ is similar to a block diagonal matrix having $-[\mathcal D + U(q)]$ and $\Id$ as diagonal blocks. Since
\[
\mathcal D + U(q)=  \trasp{A}\big( D^2 U(q) + U(q)\big) A\Big|_{\R^{4n-4-k}}=
\trasp{A}H(q) A\Big|_{\R^{4n-4-k}}
\]
where $A$ is matrix given in Proposition~\ref{thm:new-coordinates}, we infer that the eigenvalues of $B_3^*$ are given by minus the $4n-4-k$ eigenvalues of $H(q)$\footnote{Here with slight abuse of notation we denote the restriction of $H(q)$ to the $(4n-4-k)$-dimensional subspace of $T_q \mathbb S$ obtained after the symplectic decomposition again with $H(q)$.} (the Hessian of $\widehat U$ at the considered 
central configuration), and $4n-4-k$ eigenvalues equal 1. As in the previous section, $k=4$ in case of a non-planar (RE) or of a planar (RE) defined by a collinear central configuration and $k=6$ otherwise. 

Recall that the Hessian $H(q)$ is defined on the $(4n-5)$-dimensional tangent space $T_q \mathbb S$ and that 
$$\nullity{H(q)} \ge \dim T_q( \SO(4)\cdot q) = \left \{ \begin{array}{l} 6 \quad \text{if} \ q \ \text{non-planar}, \\ 5 \quad \text{if} \ q \ \text{planar non-collinear}, \\ 3 \quad \text{if} \ q \ \text{collinear}, \end{array}\right.$$
and that equality holds if $q$ is Morse-Bott non-degenerate. By comparing dimensions, we obtain that
$$\nullity{B_3^*} \ge \left \{ \begin{array}{l} 3  \quad \text{if} \ \gamma \ \text{non-planar, and } q \ \text{non-planar},\\
								2  \quad \text{if} \ \gamma \ \text{non-planar, and } q \ \text{planar},\\
								0 \quad \text{if} \ \gamma \ \text{planar, and } q \ \text{planar}.\end{array}\right .$$
and that equality holds if $q$ is Morse-Bott non-degenerate. In other words, if $\gamma$ is a non-planar (RE), then the symmetric matrix $B_3^d$ is always degenerate. 
Roughly speaking this means that, because of the complexity of the special orthogonal group $\SO(4)$, the symplectic decomposition does not allow us to remove all the 
degeneracy coming from the rotational invariance of the problem. 
This also poses major difficulties 
when trying to relate the spectrum of $B_3^d$ and that of $JB_3^d$ via the spectral flow of a suitable path of symmetric matrices starting at $B_3^d$, 
since all known formulas for computing the spectral flow assume that the endpoints of the path are invertible. For this reason, in the next section we will prove a formula of independent interest 
that allows to compute the spectral flow of a path of symmetric matrices having degenerate starting point.


\section{Spectral flow for path of selfadjoint operators in finite dimension}\label{sec:ss-ls-ham-sys}

The spectral flow is an integer-valued homotopy invariant of paths of selfadjoint Fredholm operators introduced by Atiyah, Patodi and Singer in the seventies in connection with the \textit{eta-invariant} and \textit{spectral asymmetry}. This section is devoted to recall the basic definitions and properties of the spectral flow for paths of self-adjoint operators in finite-dimensional complex vector spaces and prove a new formula for computing the spectral flow in case of affine paths having degenerate starting point. Such a formula is of independent interest and will be needed in an
essential way in Section~\ref{sec:stability-mod-2-sf}. Indeed, the discussion in Section~\ref{subsec:Morse-Bott} implies that the paths of self-adjoint operators arising in the study of the linear stability properties of a non-planar (RE) always have degenerate starting point, independently of the defining central configuration being (Morse-Bott) non-degenerate or not, 
even after the symplectic decomposition discussed in Section~\ref{subsec:decomposition}. 


\subsection{Definition and basic properties of the spectral flow}\label{subsec:spectral-flow}

	Let $\H$ be a finite-dimensional complex Hilbert space (we shall specify its dimension when needed). We denote by $\B(\H)$ the set of all (bounded) linear operators $T : \H \to \H$ and by $\Bsa(\H)$ the subset of all (bounded) linear self-adjoint operators on $\H$. 		For any $T \in \Bsa(\H)$, we define its {\sc index} $\iMor{T}$, its {\sc nullity} $\nullity{T}$ and its {\sc coindex} $\coiMor{T}$ as the numbers of negative, null and positive eigenvalues 	respectively. 	 The {\sc signature} $\sgn(T)$ of $T$ is the difference between its coindex and its index: $\sgn(T) \= \coiMor{T} - \iMor{T}$. The {\sc extended index\/} and {\sc extended coindex\/} of $T \in \Bsa(\mathcal H)$ are defined respectively by 
\[
\extiMor{T}= \iMor{T}+ \nullity{T}\qquad  \textrm{ and } \qquad \extcoiMor{T}= \coiMor{T}+ \nullity{T}.
\]
	

The {\sc spectral flow\/} $\spfl(T_t, t \in [a,b])$ of a continuous path $T: [a,b] \to \Bsa (\H)$ is roughly speaking given by the number of negative eigenvalues of $T_a$ that become positive minus the number of positive eigenvalues of $T_a$ that become negative as the parameter $t$ runs from $a$ to $b$. In other words, the spectral flow measures the net change of eigenvalues crossing $0$ and can be interpreted as a sort of generalized signature. More precisely, we have the following definition.
			\begin{dfn}\label{def:spectralflow}
		Let $a, b \in \R$, with $a < b$, and let $T :[a,b] \to \Bsa (\H)$ be a continuous path.  The { \sc spectral flow} of $T$ on the interval $[a, b]$ is defined by
			\[
				\spfl \big( T, [a, b] \big) \= \iMor{T_a  }- \iMor{ T_b} = \extcoiMor{T_b  }- \extcoiMor{ T_a}.
	 		\]
	 		A path of operators having invertible ends will be usually referred to as {\sc admissible.\/} 
	\end{dfn}

Here we list some properties of the spectral flow that will be used in this paper. In what follows,  every path of self-adjoint operators is assumed to be continuous. 

\begin{enumerate}
	\item {\sc Normalization.\/} Let $T:[a,b] \to \mathrm{GL}^{sa}(\H)$ be a path of invertible operators. Then 
	\[
	\spfl(T_t, t \in [a,b])=0.
	\]
	\item {\sc Invariance under Cogredience.\/} Let $T:[a,b] \to 	\Bsa(\H)$. Then for any  $M: [a,b] \to \mathrm{GL}^{sa}(\H)$ we have
	\[
	\spfl(T_t, t \in [a,b])=\spfl(M^*_tT_tM_t, t \in [a,b]). 
	\]
	\item {\sc Concatenation.\/} Let  $T:[a,b] \to \Bsa(\H)$ and  $c\in[a,b]$. Then	
	\[
	\spfl(T_t, t \in [a,b])= \spfl(T_t, t \in [a,c]) + \spfl(T_t, t \in[c,b])
	\]
	\item{\sc Homotopy invariance property.\/} If $F:[0,1]\times [a,b] \to \Bsa(\H)$ is a continuous family such that such that $\dim\ker F(\cdot, a)$ and $\dim\ker F(\cdot, b)$ are both constant (thus, independent on $s \in [0,1]$), then	
	\[
	\spfl(F(1,t), t \in [a,b]) = \spfl (F(0,t), t\in [a,b]).
		\]
	\item {\sc Direct sum property\/}	For $i = 1,2$, let  
$\H_i$ be two (finite dimensional) Hilbert spaces and  $T_i: [a,b] \to \Bsa(\H_i)$ be  two  paths of self-adjoint operators. Then
\[
\spfl(T_1\oplus T_2, [a, b]) = \spfl(T_1, [a, b]) + \spfl(T_2, [a, b]).
\]
\end{enumerate}

\subsection{Spectral flow for affine paths of selfadjoint operators}

Throughout the last decades much effort has been put in trying to provide efficient ways to compute 
the spectral flow of a path of self-adjoint operators. Under suitable non-degeneracy assumptions at the ``crossing'' instants, one such a way is provided through 
the so-called {\sc crossing forms\/} \cite{RS95}. However, in the applications such additional non-degeneracy assumptions are often not satisfied resp. hard to check. For this reason, authors in 
 \cite{GPP03, GPP04} provided an explicit formula to compute the spectral flow for arbitrary real analytic paths based on the theory of {\sc partial signatures\/}. 
 
 In case of affine paths, which is the only case we will be interested in throughout the paper, such formulas significantly simplify: let $C, A \in \Bsa(\H)$ with $C$ invertible, and consider the affine path 
		$$\widetilde C : (0, +\infty) \to \Bsa(\H), \quad \widetilde {D}(s) \= sA + C.$$
		If $s_* \in (0, +\infty)$ is a (possibly non-regular) crossing instant\footnote{As the path is real analytic, crossing instants are automatically isolated.} for $\widetilde{D}$, so that $1/s_*$ is an eigenvalue of $-C^{-1}A$, then for $\eps > 0$ small enough
\begin{equation} \label{eq:sfaffine}
				\spfl \bigl( \widetilde{C}(s), [s_* - \eps, s_* + \eps] \bigr) = -\sgn \Big (\langle C\,\cdot, \cdot \rangle \bigr|_{\H_{s_*}}\Big ),
			\end{equation}
		where $\H_{s_*}$ is the generalized eigenspace of the eigenvalue $1/s^*$ of $C^{-1}A$. 
By taking into account all local contributions to the spectral flow one gets the following explicit formula for computing the spectral flow of an admissible 
affine path of self-adjoint operators
	\begin{align*}
	\spfl(\widetilde C(s), s \in [a,b])= \ -\sum_{\mathclap{\substack{s_* \in (a,b) \\ s_* \textup{ crossing}}}}\ \sgn \Big (\langle C\,\cdot, \cdot \rangle \bigr|_{\H_{s_*}}\Big ).
	\end{align*}
Such a formula can be applied only if the path is admissible, hence in particular only if $s_*=a$ is not a crossing instant. As in the cases we will be interested in this will often not be the case, we provide below a way to compute the spectral flow of an affine path having a crossing instant at the starting point. 

Thus, let $S\in \Bsa(\mathcal H)$ and $L \in \mathcal B(\mathcal H)$ be such that $SL$ is self-adjoint, and consider the affine path  of self-adjoint operators  
 $\widetilde  C : \R \to \Bsa(\mathcal H)$ defined by
			\[
				\widetilde C(t) \= SL+ t S
			\]
We stress the fact that we are not assuming here that $S$ be invertible resp. that $L$ be self-adjoint. Then for every $x, y \in \mathcal H$ and every positive integer $\ell$ we have  
\begin{equation}\label{eq:proprietaSL}
\langle SL^\ell x,y\rangle= \langle Sx,L^\ell y\rangle ,
\end{equation}
as it readily follows from the fact that $S$ and $SL$ are self-adjoint. 
The Fitting decomposition theorem implies that there exists $m \in \N$ such that 
\[
\mathcal H= \ker L^m \oplus \im L^m .
\]
Such a decomposition is actually $S$-orthogonal: for $x \in \ker L^m$ and $y \in \im L^m$, we write $y=L^m u$ for suitable $u \in \mathcal H$ and compute using~\eqref{eq:proprietaSL}
\begin{equation}\label{eq:S-orto}
\langle Sx, y\rangle=\langle Sx,L^mu \rangle=\langle SL^m x, y \rangle=0 . 
\end{equation}
We set $V\= \ker L^m$, $W\=\im L^m$, and let $\pi_1, \pi_2:\mathcal H\to \mathcal H$  be the canonical projections onto $V$ and $W$ respectively. Clearly, $V$ and $W$ are $L$-invariant subspaces 
of $\mathcal H$. We now define 
\[
S_1\= \pi^*_1 S \pi_1,\qquad  S_2\= \pi^*_2 S \pi_2, \qquad L_1\=\pi_1 L \pi_1, \qquad L_2\=\pi_2 L \pi_2,
\]
and observe that $L_1^m=0$ and $ \ker L_2=V$. For $i=1,2$, let us consider the affine paths $\widetilde C_i:[a,b] \to \Bsa(\H)$ pointwise defined by
\[
\widetilde C_1(t)= S_1L_1+ t S_1 \quad \textrm{ and } \quad \widetilde C_2(t)= S_2L_2+ t S_2.
\]
Given  $x_1,x_2 \in V,  y_1,y_2 \in W$, we compute
\begin{align}
\langle \widetilde C(t)(x_1+y_1), x_2+ y_2 \rangle	& = \langle (SL+tS)(x_1+y_1), x_2+ y_2 \rangle\nonumber \\
										& = \langle (SL+tS)x_1, x_2\rangle+ \langle (SL+tS)y_1,  y_2 \rangle \label{eq:conti} \\ 
										& + \langle (SL+tS)x_1, y_2 \rangle+ \langle (SL+tS) y_1, x_2 \rangle. \nonumber 
\end{align}
Using $x_i= \pi_1 x_i$ and $y_i=\pi_2 y_i$ we get for the first term in the (RHS) of~\eqref{eq:conti}
\begin{align*}
	 \langle (SL+tS)x_1, x_2\rangle & =  \langle (SL+tS)\pi_1 x_1, \pi_1 x_2\rangle \\
	 						& = \langle \pi_1^*(SL+tS)\pi_1 x_1,  x_2\rangle\\ 
							& = \langle \pi_1^*SL\pi_1 x_1,  x_2\rangle+t \langle \pi_1^*S\pi_1 x_1, x_2\rangle\\
							& = \langle \pi_1^*S\pi_1\pi_1L\pi_1 x_1,  x_2\rangle+ t\langle \pi_1^*S\pi_1 x_1, x_2\rangle\\
							& = \langle S_1L_1 x_1, x_2 \rangle + t \langle S_1 x_1 , x_2\rangle \\
							& = \langle \widetilde C_1(t) x_1, x_2\rangle
\end{align*}
Similarly, for the second term in the (RHS) of~\eqref{eq:conti} we compute
$$ \langle (SL+tS)y_1, y_2\rangle= \langle \widetilde C_2(t) y_1, y_2\rangle .$$
The last two terms in the (RHS) of~\eqref{eq:conti} vanish: indeed, for the third term (being the argument for the fourth one completely analogous) we compute using~\eqref{eq:proprietaSL},~\eqref{eq:S-orto}, and $y_2=L^m u_2$ for some $u_2\in \mathcal H$
\begin{align*}
\langle (SL+tS)x_1, y_2 \rangle & = \langle SL x_1, L^m u_2 \rangle + t \langle S x_1, L^m u_2 \rangle = \langle SL^{m+1} x_1, u_2 \rangle + t \langle SL^m x_1, u_2 \rangle= 0. 
\end{align*}
In conclusion we have shown that
\begin{equation}
\langle \widetilde C(t)(x_1+y_1), (x_2+y_2)\rangle 
= \langle \widetilde C_1(t) x_1,x_2\rangle+\langle \widetilde C_2(t) y_1,y_2\rangle
\end{equation}
and hence,  by the additivity of the spectral flow under direct sum, 
\[
\spfl(\widetilde C(t), t \in [a,b])= \spfl(\widetilde C_1(t), t \in [a,b])+ \spfl(\widetilde C_2(t), t \in [a,b]), \quad \forall [a,b]\subset \R.
\]
The next result provides an explicit formula to compute the contribution to the spectral flow provided by the crossing instant $t=0$. It is worth noticing that Item 3 in the proposition below 
is precisely~\eqref{eq:sfaffine}. Here we therefore give an alternative proof of~\eqref{eq:sfaffine} which does not  rely on the theory of partial signatures.

\begin{prop}\label{lem:homotopia-zero}
		There exists $\varepsilon>0$ sufficiently small  such that the following hold:
		\begin{enumerate} 
			\item $\spfl(\widetilde C(t), t \in [\alpha, \beta ])= \spfl(\widetilde C_1(t), t \in [\alpha, \beta ])$ for every $[\alpha, \beta ] \subset [-\varepsilon, \varepsilon]$.
			\item  $\spfl(\widetilde C_1(t), t \in [0 ,\varepsilon ])= \iMor{S_1L_1}-\iMor{S_1}$, and $\ \spfl(\widetilde C_1(t), t \in [-\varepsilon,0 ])= \coiMor{S_1}- \iMor{S_1L_1}$.			
			\item  $\spfl(\widetilde C_1(t), t \in [-\varepsilon,\varepsilon ])= \sgn(S_1)$.
		\end{enumerate}
\end{prop}
Before proving this result, we point out that, since $S_1L_1= \pi_1^*SL\pi_1$ and $S_1=\pi_1^*S\pi_1$, then we get 
\[
\iMor{S_1L_1}=\iMor{(SL)|_V} \quad \textrm{ and } \quad 
\iMor{S_1}=\iMor{S|_V}
\]
\begin{proof}
\begin{enumerate}
\item Clearly, it is enough to prove that there exists $\varepsilon>0$ such that 
\[
\spfl(\widetilde C_2(t), t \in  [-\varepsilon ,\varepsilon ])=0.
\]
Since $L_2$ is invertible and being invertible is an open condition, we can find $\varepsilon>0$ sufficiently small such that $L_2+ t \Id$ is invertible for $|t| \le \varepsilon$. Now the identity
$$\widetilde C_2(t)= S_2L_2+ t S_2= S_2\big(L_2 + t \Id\big)$$ 
implies that $\dim \ker \widetilde C_2(t)=\dim\ker S_2$ for every $t \in [-\varepsilon, \varepsilon]$. In particular, we get 
\[
\dim \ker \widetilde C_2(0)= \dim \ker S_2L_2= \dim \big( L_2^{-1}(\ker S_2)\big)= \dim \ker S_2
\]
since $\ker L_2=V\subset \ker S_2$. By the stratum homotopy invariance of the spectral flow we infer that $\spfl(\widetilde C_2(t), t \in  [\alpha, \beta])=0$ for every $ [\alpha, \beta]\subset [-\varepsilon ,\varepsilon ]$, as claimed. 
\item Let us consider the two parameters continuous family of self-adjoint operators
\[
\widetilde h:[0,1]\times [-\varepsilon ,\varepsilon ] \to \Bsa(\mathcal H), \qquad \widetilde h(r,t)\=rS_1L_1+ t S_1.
\] 
We observe that $rL_1 + \varepsilon \Id$ is invertible since $L_1$ is nilpotent. Hence 
\[
\ker \big(r S_1 L_1 + \varepsilon S_1\big)= \ker \big(S_1(r L_1+\varepsilon \Id)\big) 
\]
implies that $\dim \ker \widetilde h(r,\varepsilon)=\dim \ker S_1$ for every $r \in [0,1]$. The stratum homotopy invariance of the spectral flow yields now that 
\[
\spfl\big(\widetilde h(r,\varepsilon), r \in [0,1]\big)=0.
\]
Since the rectangle $[0,1] \times [0,\varepsilon]$ is contractible, the homotopy invariance of the spectral flow yields
$$\spfl \big (\widetilde h(0,t), t\in [0,\epsilon)\big ) -\spfl \big (\widetilde h(1,t), t \in [0,\epsilon)\big ) - \spfl \big (\widetilde h (r,0), r \in [0,1]\big ) =0$$
which can be equivalently rewritten as 
\[
\spfl\big(\widetilde C_1(t), t \in  [0 ,\varepsilon ]\big)=\spfl\big( tS_1, t \in  [0 ,\varepsilon ]\big) -\spfl\big( rS_1L_1,  r \in  [0 ,1 ]\big) 
\]
Since by definition we have $\spfl\big( tS_1, t \in  [0 ,\varepsilon ]\big)=-\iMor{S_1}$ and $\spfl\big( rS_1L_1,  r \in  [0 ,1 ]\big)=-\iMor{S_1L_1}$, the first formula follows. 
The proof of the second formula is completely analogous and will be omitted. 
\item Follows from Item 2 using the concatenation property of the spectral flow. \qedhere
\end{enumerate}
\end{proof}
	

\section{Detecting stability through the mod 2 spectral flow}
	\label{sec:stability-mod-2-sf}
	
The aim of this section is to investigate the relation intertwining the spectrum of a symmetric $2p\times 2p$-matrix $A$ and that of the corresponding Hamiltonian matrix $JA$, where $J$ is the standard complex structure in $\R^{2p}$. Even if in general it is extremely hard to relate the two spectra, we can provide via the spectral flow modulo 2 some conditions on the spectrum of $A$ which ensure that the spectrum of $JA$ be not purely imaginary. This information will be crucial in Section~\ref{sec:instability} for the study of the stability properties of (RE). 

In what follows we denote by $G:=i J$ the {\sc Krein matrix\/} and by $\langle G\cdot,\cdot\rangle$ the {\sc Krein form\/}. Our main reference for the general theory of the Krein form and its signature is \cite[Chapter 1]{Abb01}. 

	\begin{thm} \label{prop:mainB}
	Let $A\in M(2p\times 2p,\R)$ be a symmetric matrix such that the spectrum of $JA$ is purely imaginary, i.e. such that $\sigma(JA) \subset i\R$. Then  
		\[
	\iMor{A\big|_{\ker (JA)^{2p}}}
				\equiv\iMor{A}
				\quad \mod 2.
	\]

			\end{thm}
			
\begin{rem}
\label{rmk:nocounterexample}
The kernel of $JA$ is clearly equal to the kernel of $A$, being $J$ invertible. However, this is in general not the case for the corresponding generalized eigenspaces. Indeed, 
$\ker A = \ker A^{2p}$ since $A$ is diagonalizable, but in general $\ker JA \subsetneq \ker (JA)^{2p}$. Take for instance $p=3$ and
$$A=\diag (-2,-1,1,-1,0,0).$$
It is readily seen that the characteristic polynomial of $JA$ is 
$$p_{JA}(t) = t^4 (t^2+2),$$
which means that the generalized eigenspace $\ker (JA)^{6}$ has dimension 4. Moreover, the spectrum of $JA$ is purely imaginary and $n^-(A)=3$.
Such an example appears in \cite{DZ21} and provides a counterexample to the following statement claimed in \cite[Theorem 3.11]{BJP14}: 

\begin{center}
\textit{If the spectrum of $JA$ is purely imaginary, then $n^{-}(A)$ is even.}
\end{center}

In Theorem~\ref{prop:mainB} above we fix the gap in the proof of the aforementioned theorem in \cite{BJP14} by taking into account the contribution of the generalized eigenspace corresponding to the eigenvalue zero. 

We shall finally notice that the example above does not contradict Theorem~\ref{prop:mainB}. Indeed, writing $A=\diag (A_1,A_2)$ with 
$$A_1 = \diag (-2,-1,1), \quad A_2 = \diag (-1,0,0),$$
we see that 
$$(JA)^2 =\diag (-A_1A_2, -A_1A_2),$$
and hence 
\begin{align*}
\ker (JA)^{6}=\ker (JA)^2
					&= \text{span}\, \{e_2,e_3,e_5,e_6\},\end{align*}
					where $\{e_j\}$ denotes the standard basis of $\R^{6}$. From this we conclude that 
					$$n^{-} \Big (A\Big |_{\ker (JA)^{6}}\Big ) = 1,$$
					and hence the condition $n^-\Big (A\Big |_{\ker (JA)^{6}}\Big )\equiv n^-(A)$ modulo two is satisfied. 
 \qed
\end{rem}

	\begin{proof}
		Let $D:[0,+\infty) \to \R^{2p}$ be the affine path of self-adjoint matrices pointwise defined by
			\[
				D(t) = A +tG = -J(JA - itI).
			\]
		Since $-J$ is invertible, we have
			\[
				\ker D(t) = \ker (JA - itI) \qquad \forall\,  t \in [0, +\infty).
			\]
			In particular, $t_* \in [0, +\infty)$ is a crossing instant for $D$ if and only if
				$it_* \in \sigma(JA)$. In other words, 
there is a bijection between the set of non-negative crossing instants of $D$ and the set of purely imaginary eigenvalues of $JA$ with non-negative imaginary part. As already observed the crossing instants are isolated since $D$ is affine (hence real-analytic).
					
		Let us first examine the strictly positive crossings. For a crossing instant $t_*\in (0,+\infty)$ we can find $\varepsilon >0$ such that $[t_*-\varepsilon,t_*+\varepsilon]$ does not contain 
		any other crossing instant and 
 		\begin{equation}\label{eq:Krein}
			\spfl \bigl( D(t), t\in [t_* - \eps, t_* + \eps] \bigr) = \sgn \Big (\langle G\, \cdot, \cdot \rangle \big |_{ \H_{t_*}}\Big ),
			\end{equation}
		 where $\H_{t_*}:=\ker (GA+ t_* I)^{2p}$. Indeed, the identity
		\[
		D(t) = A+tG= G(GA+t_* \Id)+(t-t_*)G,
		\]
		implies that 
		\[
		\spfl \bigl( D(t), t \in  [t_* - \eps, t_* + \eps] \bigr) = \spfl ( \widetilde{D}(t), t \in  [-\varepsilon,\varepsilon]),
		\]
		where $\widetilde{D}(t):=D(t+t_*)$, and the claim follows from Proposition~\ref{lem:homotopia-zero}, Item 3, setting  $S=G$ and $L =GA+t_*\Id$. 
	        Now, \eqref{eq:Krein} implies that 
			\begin{equation} \label{eq:sfisolcross}
				\spfl \bigl( D(t), t \in [t_* - \varepsilon, t_* + \varepsilon] \bigr) \equiv \dim \H_{t_*} \quad \mod 2,
			\end{equation}
			as it is well-known that the restriction $\langle G\, \cdot, \cdot \rangle|_{\H_{t_*}}$ is non-degenerate (for further details we refer to \cite[Chapter 1]{Abb01}).

	Clearly, we can find $\bar T>0$ such that all crossing instants are contained in $[0,\bar T)$. In particular, the additivity property of the spectral flow, together with
	the fact that zero is an isolated crossing instant and the fact that $n^-(D(T)) = n^{-}(G)=p$ for all $T\ge \bar T$, implies that 
	\begin{equation}
	\spfl \big (D(t), t\in [\epsilon, T] ) = n^-(D(\varepsilon))- p\equiv \sum \dim \mathcal H_{t_*}  \qquad \text{mod} \ 2, \quad \forall \, T\ge  \bar T,
	\label{eq:finalespfl2}
	\end{equation}
	where the sum is taken over all positive crossing instants and $\varepsilon>0$ is chosen in such a way that $[0,\varepsilon]$ does not contain any crossing instant other than $t_*=0$.
		
		When turning our attention to the crossing instant $t_*=0$, we notice that applying Item 2 of Proposition~\ref{lem:homotopia-zero} to the path $t\mapsto D(t)$ on $[0, \varepsilon]$ we obtain
			\begin{equation}\label{eq:finalespfl4}
			\begin{split}
			\spfl \bigl( D(t), t\in [0, \varepsilon] \bigr)&= n^-\bigl(A\big|_{\ker (JA)^{2p}}\bigr)-n^-\bigl(G\big|_{\ker (JA)^{2p}}\bigr) = n^-\bigl(A\big|_{\ker (JA)^{2p}}\bigr)-\dfrac{\dim \mathcal H_0}{2},
			\end{split}
						\end{equation}
			 where the second equality follows by the general theory on Krein forms (notice that $\mathcal H_0$ has even dimension). By the very definition of the spectral flow this implies that
			\begin{equation}
			\begin{split}
				\iMor{A}- \iMor{D(\varepsilon)}=\iMor{A\big|_{\ker (JA)^{2p}}}-\dfrac{\dim \mathcal H_0}{2}
			\end{split}
			\end{equation}
which is equivalent to saying that 
\begin{equation}\label{eq:92}
\iMor{D(\varepsilon)}= \iMor{A}-	n^-\bigl(A\big|_{\ker (JA)^{2p}}\bigr)+\dfrac{\dim \mathcal H_0}{2}.
\end{equation}
Comparing~\eqref{eq:finalespfl2} and~\eqref{eq:92} we obtain
			\begin{equation}\label{eq:112}
				\spfl \bigl( D(t), [\eps,T] \bigr) \equiv \iMor{D(\eps)}-p\equiv \iMor{A} -	n^-\bigl(A\big|_{\ker (JA)^{2p}}\bigr) +
 \dfrac{\dim \H_0}{2}-p\quad \mod 2
			\end{equation}
On the other hand, we have
			\[
				2p = \quad 2\ \sum \ \dim \H_{t_*} + \dim \H_0,
			\]
which can be equivalently rewritten as
			\begin{equation}\label{eq:9piu2}
				p - \dfrac{\dim \H_0}{2} = \quad\ \sum\ \dim \H_{t_*}.
			\end{equation}
			As above, the sums are here taken over all positive crossing instants. 
		Comparing  \eqref{eq:finalespfl2} and \eqref{eq:9piu2} yields
			\begin{equation}\label{eq:102}
				\spfl \bigl( D(t), [\eps,T] \bigr) \equiv p - \dfrac{\dim \H_0}{2} \equiv \frac{\dim \H_0}{2} -p \quad \mod 2.
			\end{equation}
					Finally, putting  \eqref{eq:112} and \eqref{eq:102} together  we conclude that 
			\[
				\iMor{A}-	n^-\bigl(A\big|_{\ker (JA)^{2p}}\bigr) 
				\equiv 0 
				\quad \mod 2
			\]
		as claimed. 
		 \end{proof}

As an immediate consequence of Theorem~\ref{prop:mainB} we obtain the following

\begin{cor}
\label{cor:mainB}
Let $A\in M(2p\times 2p,\R)$ be a symmetric matrix such that $\sigma (JA)\subset i \R$. If there are no non-trivial Jordan blocks for $JA$ corresponding to the eigenvalue zero (in particular, if $JA$ is diagonalizable), then
$$n^0(A) \equiv n^-(A) \equiv 0 \quad \mod \ 2.$$
\label{cor:mainB}
\end{cor}
\begin{proof}
By assumption  $\ker(JA)^{2p}=\ker(JA)= \ker A$. In particular, $n^0(A)=n^0(JA)$ is even by the spectral properties of Hamiltonian matrices,
$$n^- \Big ( A\Big |_{\ker (JA)^{2p}} \Big ) = \iMor{A\big|_{\ker A}}=0,$$ 
and hence
$$n^-(A) \equiv \iMor{A\big|_{\ker A}}= 0 \quad \text{mod} \ 2$$
as claimed.
			 \end{proof}

In the particular case in which $JA$ is invertible and diagonalizable, Corollary~\ref{cor:mainB} is proved in \cite{HS09}. 
Under such assumptions the proof is much less involved as all crossing instants are regular, 
and hence the local contributions to the spectral flow can be easily computed using crossing forms. 

\begin{rem}
Corollary~\ref{cor:mainB} implies that if $A$ is a symmetric matrix with odd dimensional kernel such that $JA$ is spectrally stable (i.e. $\sigma (JA)\subset i\R$), then the generalized eigenspsace $\ker (JA)^{2p}$ must be strictly bigger than $\ker JA$. In other words, the Jordan block structure of $JA$ corresponding to the eigenvalue zero is non-trivial and hence in particular $JA$ is linearly unstable. As a consequence, an Hamiltonian matrix $JA$ defined by a symmetric matrix $A$ having odd-dimensional kernel can never be linearly stable.

On the other hand, we can easily construct examples of spectrally stable $(JA)$ such that $A$ has odd-dimensional kernel. Consider for instance the matrix 
$$A= \diag (-2,1,1,-1,1,0).$$
Since
$$(JA)^2 = \diag (-2,-1,0,-2,-1,0)$$
we have that $\sigma (JA) = \{0,\pm i,\pm i\sqrt{2}\}$, in particular $JA$ is spectrally stable. Notice that the condition 
$$n^-\Big (A|_{\ker (JA)^6}\Big ) \equiv n^-(A) \quad \text{mod}\ 2$$
is satisfied since $n^-(A)=2$ and $\displaystyle n^-\Big (A|_{\ker (JA)^6}\Big )=0$. 
\qed
\end{rem}


%

%

%

\section{Linear and spectral instability}
\label{sec:instability}

Given a central configuration $q$ we consider the associated (RE) $\gamma$ and the corresponding linear autonomous Hamiltonian system given by
\begin{equation}\label{eq:lin-ham-sys}
	\dot \zeta(t) = -JB_*\zeta(t)
\end{equation}
where $B_*$ stands for the symmetric matrix $B_d$ defined in Equation~\eqref{eq:Ham-lin-2} respectively for the symmetric matrix $B_s$ defined in Equation~\eqref{eq:Ham-lin-pl-2} according whether the (RE) $\gamma$ is non-planar or planar.  The fundamental solution of~\eqref{eq:lin-ham-sys} can be written explicitly as 
\[
\gamma(t)\= \exp( -t JB_*).
\]

 As already observed, the invariance of Newton's equation~\eqref{eq:Newton-cpt} under translations 
and rotations yields to several integral of motions. As shown in Section~\ref{subsec:decomposition}, using such integral of motions it is possible to provide a symplectic decomposition
 of the phase space into symplectic subspaces which are invariant for the linearized Hamiltonian dynamics in~\eqref{eq:lin-ham-sys}, thus ruling out the somehow uninteresting part of the dynamics 
 due to the symmetries. Indeed, after the symplectic change of coordinates given in Lemma~\ref{thm:new-coordinates}, the matrix $B_*$ can be decomposed as the direct sum 
\[
B_*\,= B_1^*\, \oplus B_2^*\, \oplus B_3^* ,
\]
where  $B_j^*\=B\,|_{E_j^*}$, $j=1,2,3$.
Under the notation above, the symmetric matrix $B_3^*$ detects the stability properties of the (RE) $\gamma$ and takes the form
\[
B_3^*\, = \begin{pmatrix}
	-\mathcal D & \trasp{K_*}\\
	K_* & \Id
\end{pmatrix}, \qquad \mathcal D= \big[\trasp{A} D^2 U(Q) A\big]\Big|_{\R^{4n-4-k}},
\]
where $k=4 $ or $k=6$ depending on whether the (RE) is non-planar or planar but generated by a collinear central configuration, or planar and generated by a planar non-collinear central configuration, and where $Q$ is the central configuration (in the rotating frame) generating the (RE).
For this reason, we are allowed to call the (RE) $\gamma$ {\sc spectrally stable} if the Hamiltonian matrix $JB_3^*$ is spectrally stable, that is if its spectrum is purely imaginary, and {\sc linearly stable} if 
$JB_3^*$ is linearly stable, namely spectrally stable and diagonalizable. 

As shown in Section~\ref{subsec:Morse-Bott}, the matrix $B_3^*$ is similar to the block-diagonal matrix. It is easy to check that
\begin{equation}\label{eq:NsimB}
N_3^*:=
\begin{pmatrix}
-\big[\mathcal D + U(Q)\big]	 & 0 \\
 0 & \Id 
\end{pmatrix},
\end{equation}
with the matrix $\mathcal D + U(Q)$ having the same inertia indices as $H(Q)\big|_{\R^{4n-4-k}}$. 
We are now ready to state and prove the first main result of this section. 

\begin{thm}\label{thm:main-1}
Let $\gamma$ be the (RE) generated by a central configuration $q$. If:
\begin{itemize}
\item $\gamma$ is either non-planar or planar and $q$ is collinear, and we have
\[
\iMor{B_3^d\big|_{\ker (JB_3^d)^{8n-16}}} 
				-\iMor{B_3^d}\equiv 1 
				\quad \mod 2,
\]
\item or $\gamma$ is planar and $q$ is planar non-collinear, and we have
 \[
\iMor{B_3^s\big|_{\ker (JB_3^s)^{8n-20}}} 
				-\iMor{B_3^s}\equiv 1 
				\quad \mod 2,
\]
\end{itemize}
then $\gamma$  is spectrally unstable. 
\end{thm}
\begin{proof}
We prove the claim only when the first condition holds, being the other proof completely analogous. 		Let $\H \= \C^{8n-16}$ be the complexification of $\R^{8n-16}$, and define 
the path $D_3 : [0, +\infty) \to \Bsa(\H)$ as
			\[
				D_3(t) \= B_3^d + tG
			\]
		with $G \= iJ$ denotes the Krein matrix. Arguing by contradiction, we assume that $\gamma$  is spectrally stable. Applying Theorem~\ref{prop:mainB} to the path $t \mapsto D_3(t)$ (with $p=4n-8$),       we get that 
	\[
	n^-\bigl(B_3^d\big|_{\ker (JB_3^d)^{8n-16}}\bigr) 
				-\iMor{B_3^d}\equiv 0 
				\quad \mod 2.
	\]
This completes the proof. 
	\end{proof}

In the following theorem we collect some easy consequences of Theorem~\ref{thm:main-1} which relates the linear and spectral stability of a (RE) 
with the inertia indices of the corresponding central configuration. 
	
\begin{thm} \label{thm:main-2}
Let $q$ be a central configuration, and let $\gamma$ be the corresponding (RE). Then the following hold:
\begin{enumerate}
\item If $q$ is non-planar and Morse-Bott non-degenerate, or arbitrary and Morse-Bott degenerate with $n^0(q)$ even, then $\gamma$ is linearly unstable. 
\item Suppose that 
$$n^-\Big (B_3^* \Big |_{\ker (JB_3^*)^{8n-8-2k}}\Big )\equiv 0 \quad \text{mod}\ 2,$$
where $*$ and $k$ are as in Section~\ref{subsec:decomposition}. If 
$$n^-(q) + n^0(q) \equiv 0 \quad \text{mod}\ 2,$$
then $\gamma$ is spectrally unstable. In particular, if $q$ is Morse-Bott non-degenerate then:
\begin{enumerate}
\item If $q$ is non-planar and $n^-(q)$ is even, then $\gamma$ is spectrally unstable. 
\item If $q$ is planar and $n^-(q)$ is odd, then $\gamma$ is spectrally unstable.
\end{enumerate}
\end{enumerate}
\end{thm}

\begin{rem}
The assumption in Item 2 of Theorem~\ref{thm:main-2} is satisfied for instance if:
\begin{itemize}
\item $q$ is a Morse-Bott non-degenerate planar central configuration generating a planar (RE) $\gamma$. 
\item $q$ is a degenerate planar central configuration and the Hamiltonian matrix $JB_3^*$ (where $*=d$ if $\gamma$ is non-planar and $*=s$ otherwise) 
does not have non-trivial Jordan blocks corresponding to the eigenvalue zero (in this case $n^0(q)$ must be necessarily odd). 
\end{itemize}
Under the former condition we retrieve in particular the main result of \cite{HS09}. In case of planar (RE), the latter one must be added to the assumptions of the main result 
in \cite{BJP14} to conclude spectral instability.

The result above shows that the stability properties of (RE) are sensitive to the ``dimension''
of the corresponding central configurations. Indeed, all Morse-Bott non-degenerate non-planar central configurations are linearly unstable.
Moreover, under the assumption of Item 2, we see that different parities of the Morse index $n^-(q)$ force spectral instability according to the fact 
that $q$ is planar or not.  In particular, for $n\ge 4$ all (necessarily non-planar) local minima of $\widehat U$ satisfying the assumption in Item 2 are spectrally unstable. 
\qed
\end{rem}

\begin{proof}
$\ $

\begin{enumerate}
\item Under the given assumptions we readily see that 
$$n^0(B_3^d) = n^0(q)-3 \equiv 1 \quad \text{mod}\ 2, \qquad n^0(B_3^s) = \left \{\begin{array}{l} n^0(q) - 5 \\ n^0(q)-3 \end{array}\right. \equiv 1 \quad \text{mod}\, 2.$$
Therefore, 
$$\ker (JB_3^d) \subsetneq \ker (JB_3^d)^{8n-16},$$
as the generalized eigenspace must have even dimension. In particular, the matrix $JB_3^d$ has a non-trivial Jordan block structure (corresponding to the eigenvalue zero) and as such is linearly unstable. 
\item Assume that $\gamma$ is spectrally stable. Then the assumption implies in virtue of Theorem~\ref{thm:main-1} that $n^-(B_3^*)$ is even. In particular, we have that 
\begin{align*}
 4n-5 - n^-(q) - n^0(q) = n^+(q) = n^-(B_3^*) \equiv 0 \quad \text{mod}\ 2,
\end{align*}
which implies that 
$$n^-(q) + n^0(q) \equiv 1 \quad \text{mod}\ 2,$$
showing the first claim. To conclude (a) and (b) it suffices to  notice that in the first case $n^0(q)=6$ is even, whereas in the second one $n^0(q)=3,5$ depending on $q$ being collinear or planar not-collinear. 
\qedhere
\end{enumerate}
\end{proof}

%

\appendix

\section{Symmetries and integrals of motion}\label{subsec:integrals}

The $n$-body problem in $\R^4$ is invariant under the action of the $4$-dimensional Euclidean group $\Euc{4}$. By lifting such a group action to the phase space we can
use Noether's theorem to foliate the space with sets which are invariant under the Hamiltonian dynamics of the $n$-body problem. The goal of this appendix is to provide 
an explicit description of the differentiable structure of such  invariant sets. 

Let $g \in \Euc{4}$ be an element of the Euclidean group; so, $g:\R^4 \to \R^4$ takes the form 
\[
g\cdot q = Aq+ b \qquad A \in \OO(4) \textrm{ and } b \in \R^4.
\]
By lifting such an action to $T^*\X $ via 
\[
g\cdot (q,p)=(g \cdot q, p  \trasp{A}) 
\]
it is straightforward to check that both the Hamiltonian function $H$ as well as the Liouville one-form are invariant under the $g$-action.  For fixed $b \in \R^4$ we have a one-parameter family 
\[
g_b^s= \Id + s b
\]
whose generating vector field is 
\[
X_b\,\equiv b. 
\]
Noether's theorem implies that the quantity 
\[
F(q,p)\= p [X_b]= \left(\sum_{i=1}^n p_i\right)[b]
\]
is   a first integral of  motions for every $b \in \R^4$, that is, that the {\sc total linear momentum\/} 
\begin{equation}
	\overline p\= p_1 + \ldots p_n=\sum_{i=1}^n p_i \in (\R^4)^*
\end{equation}
is constant and determines the motion of the {\sc center of mass\/} 
\[
\overline q\= \dfrac{1}{\overline m}\sum_{i=1}^n m_i q_i, \qquad \overline m \= \sum_{i=1}^n m_i,
\]
which is namely given by 
\[
\overline q(t)= \overline q(t_0) + \overline p (t-t_0).
\]
If $t\mapsto \big(q(t), p(t)\big)$ is any solution of the $n$-body problem having center of mass at $\overline q$ and  momentum $\overline p$, then   
\[
\widehat q(t)\= q(t)-\overline q \qquad \widehat p(t)\= p(t)-\overline p
\]
is another solution having total momentum zero.   
Thus, one can without loss of generality study solutions with $\overline {p}=0$.  This implies that $\overline q$ is a vector-valued constant of motion. 
Hence it is not restrictive to assume $\overline q$=0. Summing up, the translational invariance of the $n$-body problem yields 8 integrals of motion. 

Let's now choose $1 \le k <l \le 4$ (so, only 6 choices are possible) and let us consider the one-parameter group of rotations in the $(q^k,q^l)$-plane. After rearranging the coordinates as $(q^k,q^l, *,*)$, we get 
\[
g_s= 
\begin{pmatrix}
  \begin{matrix}
  \cos s & -\sin s \\
  \sin s & \cos s
  \end{matrix}
  & \rvline & \zero \\
\hline  \zero & \rvline &
 \zero
\end{pmatrix}
\]
where $\zero$ denotes the $2 \times 2$ null matrix. The  generating vector field of $g_s$ is pointwise defined by 
\[
X(q)= \left(\dfrac{d}{ds}\Big\vert_{s=0} g^s\cdot q\right)= 
\begin{pmatrix}
  \begin{matrix}
 0 & -1 \\
  1 & 0
  \end{matrix}
  & \rvline &\zero \\
\hline  \zero & \rvline &
 \zero
 \end{pmatrix}\begin{pmatrix}
q^k\\ q^l\\0\\0
\end{pmatrix}=\begin{pmatrix}
-q^l\\ q^k\\0\\0
\end{pmatrix}
\]
Using Noether's theorem again we obtain (after rearranging the coordinates of $p$ accordingly) that the quantity 
\[
F(q,p)\= \sum_{i=1}^n p_i [X(q_i)]=\sum_{i=1}^n \big( -p_i^k[q_i^l]+ p_i^l[q^k_i]\big)
\]
is a first integral of  motion. Thus we get 6 additional first integrals 
\[
\Omega_{kl}, \qquad 1 \le k <l \le 4,
\]
which will be referred to as the {\sc angular momentum first integrals.\/} 

Summarizing, we have shown that for fixed constants 
$\omega_{kl}$, $1 \le k <l \le 4 $, the set 
\[
S_{\omega_{kl}}\=\Big \{ (q,p)\in T^*\X\ \Big  |\  \overline q= \overline p=0, \ \Omega_{kl}(q,p)=\omega_{kl}\Big \}
\]
is invariant under the Hamiltonian dynamics of the $n$-body problem. We will now investigate under which condition the set $S_{\omega_{kl}}$ is a smooth $(8n-14)$-dimensional manifold.

\begin{prop}
The 14 integrals of motion coming from the conservation of total linear momentum, center of mass and angular momentum are linearly independent except for the case in which all vectors $q_i,p_i$ for $i =1, \ldots n$ are coplanar. In such a case, the motion is planar.  
\end{prop}

\begin{proof}
To prove the claim we just have to show  that the $(8n\times 14)$-dimensional matrix defined by 
\[
\begin{pmatrix}
	\nabla \overline q \ |\ \nabla \overline p\ | \ \nabla \Omega
\end{pmatrix}
\]
has rank strictly less than $14$ at $(q,p)$ if and only if all vectors $q_i,p_i$ for $i =1, \ldots n$ are coplanar (as usual we identify $p$ with the corresponding column vector). 
By a straightforward computation, we get that 
\[
\nabla \overline q= \dfrac{1}{\overline m}\left(\begin{array}{@{}c@{}}
m_1 \Id_4 \\ \vdots \\ m_n\Id_4\\ \hline 0_4 \\ \vdots \\ 0_4
\end{array}\right), \qquad 
\nabla \overline p=\left(\begin{array}{@{}c@{}}
0_4\\ \vdots \\0_4\\ \hline 
 \Id_4 \\ \vdots \\ \Id_4
\end{array}\right),
\]
and 
\[
\nabla \Omega=\left(\begin{array}{@{}cccccc@{}}
p_2^1 & p_3^1 &  p_4^1 & 0 & 0 & 0 \\
- p_1^1 & 0 & 0 & p_3^1& p_4^1 & 0 \\
0 & -p_1^1& 0 & -p_2^1 & 0 & p_4^1\\
0 & 0&  -p_1^1 & 0 & -p_2^1 & -p_3^1\\
\hline
\vdots &\vdots & \vdots & \vdots &\vdots &\vdots \\ \hline \\
p_2^n & p_3^n &  p_4^n & 0 & 0 & 0 \\
- p_1^n & 0 & 0 & p_3^n& p_4^n & 0 \\
0 & -p_1^n& 0 & -p_2^n & 0 & p_4^n\\
0 & 0&  -p_1^n & 0 & -p_2^n & -p_3^n\\
\hline\\
-q_2^1 & -q_3^1 &  -q_4^1 & 0 & 0 & 0 \\
q_1^1 & 0 & 0 & -q_3^1& -q_4^1 & 0 \\
0 & q_1^1& 0 & q_2^1 & 0 & -q_4^1\\
0 & 0&  q_1^1 & 0 &q_2^1 & q_3^1\\
\hline
\vdots &\vdots & \vdots & \vdots &\vdots &\vdots \\ \hline \\
-q_2^n & -q_3^n &  -q_4^n & 0 & 0 & 0 \\
q_1^n & 0 & 0 & -q_3^n& -q_4^n & 0 \\
0 & q_1^n& 0 & q_2^n & 0 & -q_4^n\\
0 & 0&  q_1^n & 0 &q_2^n & q_3^n\\
\end{array}\right)= 
\left(\begin{array}{@{}c@{}}
	\rvline\\ \nabla^1 \Omega\\\rvline  \\ \hline   \rvline \\ \nabla^2 \Omega\\ \rvline  
\end{array}
\right)
\]
The vectors $\nabla \overline q$, $\nabla \overline p$, $\nabla \Omega$ are thus linearly dependent if and only if the following two systems of linear equations
\begin{equation}\label{eq:eq1}
\begin{cases}
	m_i \lambda_1 + \lambda_9 p_2^i+ \lambda_{10}p_3^i+ \lambda_{11} p_4^i=0\\
	m_i \lambda_2 - \lambda_9 p_1^i+ \lambda_{12}p_3^i+ \lambda_{13} p_4^i=0\\
	m_i \lambda_3 - \lambda_{10} p_1^i+ \lambda_{12}p_2^i+ \lambda_{14} p_4^i=0\\
	m_i \lambda_4 - \lambda_{11} p_1^i- \lambda_{13}p_2^i- \lambda_{14} p_3^i=0
\end{cases}\qquad \forall\, i =1, \ldots, n
\end{equation}
and 
\begin{equation}\label{eq:eq2}
\begin{cases}
	 \lambda_5 - \lambda_9 q_2^i- \lambda_{10}q_3^i- \lambda_{11} q_4^i=0\\
	\lambda_6 + \lambda_9 q_1^i- \lambda_{12}q_3^i- \lambda_{13} q_4^i=0\\
	 \lambda_7 + \lambda_{10} q_1^i+\lambda_{12}q_2^i- \lambda_{14} q_4^i=0\\
	 \lambda_8 + \lambda_{11} q_1^i+ \lambda_{13}q_2^i+ \lambda_{14} q_3^i=0
\end{cases}\qquad \forall\, i =1, \ldots, n
\end{equation}
have a non-trivial solution $\lambda_1,...,\lambda_{14}$.
By summing over $i$ in each line of~\eqref{eq:eq1} and~\eqref{eq:eq2} we obtain
\[
\lambda_1= \lambda_2 = \ldots= \lambda_8=0
\]
(here we use the fact that $\overline q=\overline p=0$). Moreover, for each $i=1, \ldots, n$ the vectors $q_i, p_i$ are contained in the linear solution subspace $V\subset \R^4$ of the following (linear) system 
\begin{equation}\label{eq:eq3}
\begin{cases}
 \lambda_9 x_2+ \lambda_{10}x_3+ \lambda_{11} x_4=0\\
 - \lambda_9 x_1+ \lambda_{12}x_3+ \lambda_{13} x_4=0\\
 - \lambda_{10} x_1+ \lambda_{12}x_2+ \lambda_{14} x_4=0\\
- \lambda_{11} x_1- \lambda_{13}x_2- \lambda_{14} x_3=0
\end{cases}
\end{equation}
whose coefficient matrix is
\[
\Lambda= 
\begin{pmatrix}
	0 & \lambda_9& \lambda_{10} & \lambda_{11}\\
	-\lambda_9 &0&\lambda_{12} & \lambda_{13}\\
	-\lambda_{10} & -\lambda_{12} & 0 & \lambda_{14}\\
	-\lambda_{11} & -\lambda_{13} & -\lambda_{14} & 0
\end{pmatrix}
\]
It is readily seen that $V$ is at most 2-dimensional as soon as at least one $\lambda_j$ is non-zero. This implies that the considered first integral can be linearly dependent if and only if all $q_i$ and $p_i$ are 
coplanar. 
\end{proof}

\begin{rem}
If all $q_i$ and $p_i$ are coplanar, then up to a rotation we may assume that they are all contained in  $\R^2 \times \{0\}\subset \R^4$. By definition, this implies that 
\begin{align*}
\Omega_{kl}(q,p)=0 \quad \forall\, (k,l)\neq(1,2), \qquad 
\Omega_{12}(q,p)=\begin{cases}
0 & \textrm{ all }  q_i\  \text{and} \ p_i \textrm{ are collinear}\\
\neq 0 & \textrm{ otherwise }	
\end{cases} 
\end{align*}
The case in which all vectors $q_i, p_i$ are coplanar can therefore be ruled out by assuming that 
$
\omega_{kl} \neq 0$ for at least two different pairs $(k,l)$. \qed
\end{rem}

As the Hamiltonian $H$ of the $n$-body problem is autonomous, $H$ provides the 15-th integral of motion. For $h\in \R$, the invariant set 
$$S_{\omega_{kl},h} := S_{\omega_{kl}} \cap \{H=h\}$$
is a smooth $8n-15$-dimensional manifold provided the 15 integral of motions are linearly independent. In particular, as seen above, we must have that $\omega_{kl}\neq 0$ for at least two different pairs $(k,l)$. 
In order to understand when the 15 integrals of motions are independent one should understand in which cases 
\[
\nabla H(q,p)= (\nabla \overline q\cdot u) \  (\nabla \overline p\cdot v)+ \nabla \Omega(q,p)\cdot w
\]
possesses non-trivial solutions $u,v \in \R^4$ and $w \in \R^6$. Such a condition translates into the following system 
\begin{equation}\label{eq:dip15integrals-1}
\begin{cases}
	-\nabla U(q)= \dfrac{1}{\overline m} u \begin{pmatrix}m_1 \Id_4\\ \vdots\\ m_n \Id_4\end{pmatrix} + w \cdot \nabla^1 \Omega(q,p)\\
	M^{-1} p = v \begin{pmatrix}\Id_4\\ \vdots\\  \Id_4
 \end{pmatrix}
 + w \cdot \nabla^2\Omega(q,p)
\end{cases}
\end{equation}
By summing up the  $n$ four dimensional components of the first equation in \eqref{eq:dip15integrals-1}, we get
\[
0=u+ w \cdot \underbrace{\begin{pmatrix}
 	\overline{p_2}&\overline{p_3}& \overline{p_4}&0&0&0\\
 	-\overline{p_1}&0&0&\overline{p_3}& \overline{p_4}&0\\
 	0 & -\overline{p_1}&0& -\overline{p_2}& 0&\overline{p_4}\\
 	0&0&-\overline{p_1}&0& -	\overline{p_2}&-\overline{p_3}
 \end{pmatrix}}_{=0\  (\textrm{since } \overline p=0)}=u
\]
Similarly, summing  up the  $n$ four dimensional components of the second equation in \eqref{eq:dip15integrals-1} with weights $m_i$ and since we are assuming that the total mass of the system is 1, yields
\[
0=\overline p=  v + w \cdot \underbrace{\begin{pmatrix}
 	- \overline m\,\overline{q_2}&- \overline m\,\overline{q_3}& - \overline m\,\overline{q_4}&0&0&0\\
 	 \overline m\,\overline{q_1}&0&0&- \overline m\,\overline{q_3}& - \overline m\,\overline{q_4}&0\\
 	0 &  \overline m\,\overline{q_1}&0& \overline m\, \overline{q_2}& 0&- \overline m\,\overline{q_4}\\
 	0&0& \overline m\,\overline{q_1}&0& \overline m\, \overline{q_2}& \overline m\,\overline{q_3}
 \end{pmatrix}}_{=0\   (\textrm{since } \overline q=0)}=  v \ \Rightarrow v=0.
\]
It follows that the 15 integrals of motions are linearly independent if and only if 
\begin{equation}
\nabla H(q,p)= w \cdot \nabla \Omega(q,p)
\label{eq:finalintegral}
\end{equation}
does not possess non-trivial solutions $w\in \R^6$. Understanding whether or not~\eqref{eq:finalintegral} admit non-trivial solutions seems to be quite involved. As this is not relevant 
for the present work, we rather leave it for further work.

\vspace{1cm}
\noindent
\textsc{Dr. Luca Asselle},
Justus Liebig Universit\"at Gie\ss en,
Arndtrstrasse 2,
35392, Gie\ss en,
Germany\\
E-mail: $\mathrm{luca.asselle@math.uni}$-$\mathrm{giessen.de}$

\vspace{5mm}
\noindent
\textsc{Prof. Alessandro Portaluri},
Università degli Studi di Torino,
Largo Paolo Braccini 2,
10095 Grugliasco, Torino,
Italy,
E-mail: $\mathrm{alessandro.portaluri@unito.it}$

\vspace{5mm}
\noindent
\textsc{Prof. Li Wu},
Department of Mathematics,
Shandong University,
Jinan, Shandong, 250100,
China\\
E-mail: $\mathrm{vvvli@sdu.edu.cn}$

\end{document}